\documentclass[smallextended,numbook,runningheads]{svjour3_xxx}     
\smartqed  % flush right qed marks, e.g. at end of proof
\usepackage{amsmath}
\usepackage{mathptmx}      % use Times fonts if available on your TeX system
\usepackage{amsfonts}
\usepackage{amssymb}
\usepackage{enumerate}   
\usepackage{verbatim}
\usepackage{url}
\journalname{BIT}
\newcommand\Order{\mathcal{O}}
\newcommand\nA{\mathcal{A}}
\newcommand\nS{\mathcal{S}}

\newcommand\nR{\mathcal{R}}
\newcommand\nB{\mathcal{B}}
\newcommand\nW{\mathcal{W}}
\newcommand\nE{\mathcal{E}}
\newcommand\nL{\mathcal{L}}
\newcommand{\ee}{\mathrm{e}}
\newcommand{\dd}{\mathrm{d}}
\renewcommand{\AA}{\mathtt{A}}
\newcommand{\BB}{\mathtt{B}}
\renewcommand{\gg}{\mathfrak{g}}
\newcommand{\Id}{\mathrm{Id}}
\newcommand{\grade}{\mathrm{grade}}
\newcommand{\coeff}{\mathrm{coeff}}

\begin{document}
\title{Order conditions for exponential integrators\thanks{This 
work was supported by  the Austrian Science Fund (FWF) under Grant P 30819-N32.}}

\author{Harald Hofst\"{a}tter}

\institute{Harald Hofst\"{a}tter \at
           Fakult\"{a}t f\"{u}r Mathematik, Universit\"{a}t Wien, 
           Oskar-Morgenstern-Platz~1, 1090 Wien, Austria. 
           \email{hofi@harald-hofstaetter.at}           
}

\date{\vspace{5mm}}

\maketitle

\begin{abstract}This paper provides an algebraic framework for the generation
of order conditions for the construction of exponential integrators like 
splitting and Magnus-type methods for the numerical solution of evolution equations.
The generation of order conditions is based on an analysis of the structure of the
leading local error term of such an integrator, and on a new algorithm for the
computation of coefficients of words in expressions involving exponentials. As an application a new 8th order commutator-free Magnus-type integrator involving only 
8 exponentials is derived.
%Insert your abstract here. No references or citations in abstract! 
%Include keywords and mathematical subject classification numbers as needed.
\end{abstract}

\section{Introduction}

The construction of splitting methods for the numerical integration of evolution equations
like\footnote{For simplicity we assume that here $A$ and $B$ are complex matrices. 
However, the purely formal considerations of this paper are relevant for much more general situations, e.g., using the calculus of Lie derivatives, for nonlinear evolution equations of type $\partial_t u(t)=A(u(t))+B(u(t))$.
}
\begin{equation}\label{eq:split_evolution_eq}
\partial_t u(t) = Au(t)+Bu(t),\quad t\geq 0, \quad u(0)=u_0, 
\quad A,B\in\mathbb{C}^{d\times d},
\end{equation}
or of Magnus-type integrators for non-autonomous evolution equations
\begin{equation}
\partial_t u(t) = A\label{eq:non_auto_evolution_eq}(t)u(t),\quad t\geq 0,\quad u(0)=u_0,\quad A(t)\in\mathbb{C}^{d\times d}
\end{equation}  
has been extensively studied in the literature,
see, e.g., \cite{part1} and the references therein for splitting methods,
and \cite{alvfeh11} and \cite{SergioFernandoMPaper2}  and the references therein for 
Magnus-type integrators.
In this paper we provide a general and unified algebraic framework for the systematic generation of order conditions needed for the construction of such integrators.\footnote{A Maple code for the 
automatic generation of order conditions based on ideas of this paper
 is currently being prepared for publication, see \url{https://github.com/HaraldHofstaetter/Expocon.mpl}.}
It is likely that the ideas and results of this paper can also be adapted 
to other classes of exponential integrators, which, however, are not discussed 
in this paper.

In the remainder of this introduction we give an exemplary overview of our approach,
which is then elaborated in a systematic and purely formal way in 
Section~\ref{sct:theory}. The theoretical considerations of Section~\ref{sct:theory} are then applied to an example of a generalized splitting method in Section~\ref{sct:splitting}, and
to examples of Magnus-type methods in Section~\ref{sct:magnus}. In particular, we report on the construction of a new 8th order commutator-free Magnus-type integrator 
involving only 8 exponentials in Subsection~\ref{subsct:eight}.

\subsection{Coefficients of words in expressions involving exponentials}
\label{intro:coeff}
In our approach to  the generation of order conditions we make use
of  a new algorithm for the efficient computation
of coefficients $c_w$ %=\coeff(w,X)$ 
of words $w\in\nA^*$ over an alphabet $\nA$ in the formal expansion
\begin{equation}
  X=\sum_{w\in\nA^*}c_w w
\end{equation}
of an expression $X$ involving exponentials $\ee^Y$ with exponents $Y\in\mathbb{C}\langle\nA\rangle$ being
polynomials in the non-commuting variables $\in\nA$.
Here, formally, $X$ is an element of $\mathbb{C}\langle\langle\nA\rangle\rangle$, the algebra of formal power series in the non-commuting variables $\in\nA$.
The essential ingredient of this algorithm is a family of algebra homomorphisms $\{\varphi_w: w\in\nA^*\}$ such
that for each word $w\in\nA^*$ of length $\ell(w)$, $\varphi_w(X)$ is an upper triangular matrix 
$\in\mathbb{C}^{(\ell(w)+1)\times(\ell(w)+1)}$ whose entries are coefficients 
of subwords
of $w$ in $X$,
\begin{equation*}
\varphi_w(X)_{i,j} = 
%\varphi_{i,j}^{(w)}), \
%\varphi_{i,j}^{(w)}=
\left\{\begin{array}{ll}
\coeff(w_{i:j-1},X),&\mbox{if $i<j$,}\\
\coeff(\Id,X),&\mbox{if $i=j$,}\\
0,&\mbox{if $i>j$.}
\end{array}\right.
\end{equation*}
Here $w_{i:j-1}=w_iw_{i+1}\cdots w_{j-1}$ denotes the subword of $w$ of length $j-i$, starting at position $i$ and ending at position $j-1$, and $\Id$ denotes the empty word.
Besides compatibility with operations $+$ and $\cdot$ making $\varphi_w$ an algebra homomorphism, $\varphi_w$ is also compatible with exponentiation, $\varphi_w(\ee^Y)=\exp(\varphi_w(Y))$, where
on the right the matrix exponential is exactly computable if $\varphi_w(Y)$ is a strict upper triangular matrix, which is the case if the empty word $\Id$ does not occur
in the expression $Y$.
A recursive application of $\varphi_w$ (the recursion terminates with well-defined values $\varphi_w(a)$ for the ``atoms'' $a\in\nA$) yields $\varphi_w(X)$, from which one can read off the coefficients $c_v$ for all subwords $v$ of $w$. %=\coeff(w,X)$.
A formal justification of this algorithm is provided by our 
Theorem~\ref{thm:III}.
We note that there is some similarity to the algorithms proposed in \cite{Reinsch2000,vanBruntVisser1026} for computing the coefficients 
of the Baker--Campbell--Hausdorff series.
\subsection{Applications to splitting methods}
For example for 
$\nA=\{\AA,\BB\}$, $X=\ee^{\frac{1}{2}\BB}\,\ee^{\AA}\,\ee^{\frac{1}{2}\BB}$, $w=\AA\AA\BB$
we have
\begin{align*}
\varphi_{\AA\AA\BB}(\ee^{\frac{1}{2}\BB}\,\ee^{\AA}\,\ee^{\frac{1}{2}\BB})
&=\exp\big(\tfrac{1}{2}\varphi_{\AA\AA\BB}(\BB)\big)\cdot\exp\big(\varphi_{\AA\AA\BB}(\AA)\big)\cdot\exp\big(\tfrac{1}{2}\varphi_{\AA\AA\BB}(\BB)\big)\\
&=\exp\left(\begin{array}{cccc}
0 & 0 & 0 & 0\\
0  & 0 & 0  & 0  \\
0 & 0 & 0 & \tfrac{1}{2} \\
0 & 0 & 0 & 0
\end{array}
\right)\cdot
\exp\left(\begin{array}{cccc}
0 & 1 & 0 & 0\\
0  & 0 & 1 & 0 \\
0 & 0 & 0 & 0 \\
0 & 0 & 0 & 0
\end{array}
\right)\cdot
\exp\left(\begin{array}{cccc}
0 & 0 & 0 & 0\\
0  & 0 &0 & 0  \\
0 & 0 & 0 & \frac{1}{2} \\
0 & 0 & 0 & 0
\end{array}
\right)\\
&=\left(\begin{array}{cccc}
1 & 0 & 0 & 0\\
0  & 1 &0  & 0  \\
0 & 0 & 1 &\frac{1}{2}   \\
0 & 0 & 0 & 1
\end{array}
\right)\cdot
\left(\begin{array}{cccc}
1 & 1 & \frac{1}{2} & 0\\
0  & 1 & 1 & 0 \\
0 & 0 & 1 & 0 \\
0 & 0 & 0 & 1
\end{array}
\right)\cdot
\left(\begin{array}{cccc}
1 & 0 & 0 & 0\\
0  & 1 &0  & 0  \\
0 & 0 & 1 &\frac{1}{2}   \\
0 & 0 & 0 & 1
\end{array}
\right)\\
&=\left(\begin{array}{cccc}
1 & 1 & \frac{1}{2} & \frac{1}{4}\\
0 & 1 & 1 & \frac{1}{2}\\
0 & 0 & 1 & 1 \\
0 & 0 & 0 & 1
\end{array}\right)
=\left(\begin{array}{cccc}
c_\Id & c_\AA & c_{\AA\AA} & c_{\AA\AA\BB}\\
0 & c_\Id & c_\AA & c_{\AA\BB}\\
0 & 0 & c_\Id & c_\BB \\
0 & 0 & 0 & c_\Id
\end{array}\right),
\end{align*}
from which we read off
\begin{align*}
\ee^{\frac{1}{2}\BB}\,\ee^{\AA}\,\ee^{\frac{1}{2}\BB}
&=c_{\Id}\Id +c_{\AA}\AA+c_{\BB}\BB+c_{\AA\AA}\AA\AA+c_{\AA\BB}\AA\BB+c_{\AA\AA\BB}\AA\AA\BB+\dots\\
&=\Id +\AA+\BB+\tfrac{1}{2}\AA\AA+\tfrac{1}{2}\AA\BB+\tfrac{1}{4}\AA\AA\BB+\dots\,.
\end{align*}
Similar calculations involving $\varphi_w(\ee^{\frac{1}{2}\BB}\,\ee^{\AA}\,\ee^{\frac{1}{2}\BB})$ for $w\in\{\AA\AA\AA$, $\AA\AA\BB$, $\AA\BB\AA$, $\BB\AA\AA$, $\AA\BB\BB$, $\BB\AA\BB$, $\BB\BB\AA$, $\BB\BB\BB\}$
yield
\begin{align*}
\ee^{\frac{1}{2}\BB}\,\ee^{\AA}\,\ee^{\frac{1}{2}\BB}
&=\Id +\AA+\BB+\tfrac{1}{2}\AA\AA+\tfrac{1}{2}\AA\BB+\tfrac{1}{2}\BB\AA+\tfrac{1}{2}\BB\BB\\
&\quad+\tfrac{1}{6}\AA\AA\AA
+\tfrac{1}{4}\AA\AA\BB+\tfrac{1}{4}\BB\AA\AA
+\tfrac{1}{8}\AA\BB\BB+\tfrac{1}{4}\BB\AA\BB+\tfrac{1}{8}\BB\BB\AA
+\tfrac{1}{6}\BB\BB\BB
+\dots\,,
\end{align*}
where the dots represent terms involving words of length greater than three.
By repeating these considerations for the expression $X=\ee^{\AA+\BB}$, e.g.,
\begin{equation*}
\varphi_{\AA\AA\BB}(\ee^{\AA+\BB})
=\exp\big(\varphi_{\AA\AA\BB}(\AA)+\varphi_{\AA\AA\BB}(\BB)
\big)
=\exp\left(\begin{array}{cccc}
0 & 1 & 0 & 0\\
0 & 0 & 1 & 0\\
0 & 0 & 0 & 1 \\
0 & 0 & 0 & 0
\end{array}\right)
=\left(\begin{array}{cccc}
1 & 1 & \frac{1}{2} & \frac{1}{6}\\
0 & 1 & 1 & \frac{1}{2}\\
0 & 0 & 1 & 1 \\
0 & 0 & 0 & 1
\end{array}\right),
\end{equation*}
we obtain
\begin{align*}
\ee^{\AA+\BB}
&=\Id +\AA+\BB+\tfrac{1}{2}\AA\AA+\tfrac{1}{2}\AA\BB+\tfrac{1}{2}\BB\AA+\tfrac{1}{2}\BB\BB\\
&\quad+\tfrac{1}{6}\AA\AA\AA
+\tfrac{1}{6}\AA\AA\BB+\tfrac{1}{6}\AA\BB\AA+\tfrac{1}{6}\BB\AA\AA
+\tfrac{1}{6}\AA\BB\BB+\tfrac{1}{6}\BB\AA\BB+\tfrac{1}{6}\BB\BB\AA
+\tfrac{1}{6}\BB\BB\BB
+\dots\,,
\end{align*}
and thus
\begin{align}
\ee^{\frac{1}{2}\BB}\,\ee^{\AA}\,\ee^{\frac{1}{2}\BB}-\ee^{\AA+\BB}
&=\tfrac{1}{12}\AA\AA\BB-\tfrac{1}{6}\AA\BB\AA+\tfrac{1}{12}\BB\AA\AA
-\tfrac{1}{24}\AA\BB\BB+\tfrac{1}{12}\BB\AA\BB-\tfrac{1}{24}\BB\BB\AA
+\dots\nonumber\\
&=\tfrac{1}{12}[\AA,[\AA,\BB]]-\tfrac{1}{24}[[\AA,\BB]\BB]+\dots\,.
\label{eq:strang_err_abstract}
\end{align}
Here all terms involving words of length less than three cancel,
and  the words of length three can be combined to a homogeneous Lie element, i.e., a linear combination of commutators of length  three in $[\mathbb{C}\langle\AA,\BB\rangle]$, the free Lie algebra generated by $\AA$ and $\BB$.

We can interpret (\ref{eq:strang_err_abstract}) as a result 
for Strang splitting, an example of a splitting method for the numerical solution of evolution equations of 
type (\ref{eq:split_evolution_eq}),
where one step
\begin{equation}\label{eq:step_split}
u_n\mapsto u_{n+1}=\nS(\tau)u_n
\end{equation}
of step-size $\tau$ is defined by
$$\nS(\tau)=\ee^{\frac{1}{2}\tau B}\,\ee^{\tau A}\,\ee^{\frac{1}{2}\tau B}.$$
By substituting $\AA\to\tau A$, $\BB\to\tau B$ in (\ref{eq:strang_err_abstract}) it follows
that the local error satisfies
$$\nL(\tau)=\nS(\tau)-\nE(\tau)=\tau^3\big(\tfrac{1}{12}[A,[A,B]]-\tfrac{1}{24}[[A,B],B]\big)\ +\Order(\tau^4),$$
where $\nE(\tau)=\ee^{\tau(A+B)}$ is the exact local solution 
operator.\footnote{Symbols $\AA$, $\BB$, etc.\ written in typewriter
font denote purely abstract objects. In applications of  formal
results to concrete situations we use corresponding symbols $A$, $B$, etc.~
}
This result can be generalized: 
the leading term in the local error is a homogeneous Lie element of order 
$\Order(\tau^{p+1})$ for any splitting method $\nS(\tau)$ of order $p$ 
and even for ``generalized'' splitting methods
like the fourth order method
\begin{equation}\label{eq:gs4}
\nS(\tau)=\ee^{\frac{1}{6}\tau B}\,\ee^{\frac{1}{2}\tau A}\,\ee^{\frac{2}{3}\tau B+\frac{1}{72}\tau^3[B,[A,B]]}
	\,\ee^{\frac{1}{2}\tau A}\,\ee^{\frac{1}{6}\tau B}
\end{equation}	
proposed in \cite{suzuki95,chin97}. %, see our Theorem~\ref{thm:I}.
For splitting methods this result was proved in \cite[Theorem~2.6]{auzingeretal13c}. 
Our much more general Theorem~\ref{thm:I} covers also the case of 
generalized splitting methods (and Magnus-type methods, see Section~\ref{intro:magnus}
below).

For the concrete calculation of the leading error term one has to compute
the coefficients of all Lyndon words of length $p+1$ in $\nL=\nS-\nE$
(whose number is significantly smaller than that of all words of length $p+1$), and then
apply an easily computable transformation matrix to these coefficients to get the leading error term represented in the Lyndon basis of the subspace generated by the commutators of length $p+1$. 
The representation of the leading error term as a homogeneous Lie element
together with this relation between Lyndon words and Lyndon basis elements
leads to an effective procedure for the generation of minimal sets of order conditions for (generalized)
splitting methods. We illustrate this with an example. Let
$$\nS=\ee^{b_3\BB}\,\ee^{a_3\AA}\,\ee^{b_2\BB}\,\ee^{a_2\AA}\,\ee^{b_1\BB}\,\ee^{a_1\AA},\quad \AA=\tau A, \ \BB=\tau B$$
be a splitting method with parameters $a_1, a_2, a_3, b_1, b_2, b_3$ to be determined such that $\nS$ has order $p=3$. We calculate the
coefficients $c_w$ of all Lyndon words $w$ of length less than four in 
$\nL=\nS-\ee^{\AA+\BB}$,
\begin{align*}
c_{\AA}&= a_1+a_2+a_3-1,\\
c_{\BB}&=b_1+b_2+b_3-1,\\
c_{\AA\BB}&=a_2b_1+a_3b_1+a_3b_2-\tfrac{1}{2},\\
c_{\AA\AA\BB}&=\tfrac{1}{2}a_2^2 b_1+\tfrac{1}{2}a_3^2 b_1+\tfrac{1}{2}a_3^2 b_2+a_2a_3b_1-\tfrac{1}{6},\\
c_{\AA\BB\BB}&=\tfrac{1}{2}a_2 b_1^2+\tfrac{1}{2}a_3 b_1^2+\tfrac{1}{2}
a_3b_2^2+a_3b_1b_2-\tfrac{1}{6}.
\end{align*}
If $c_{\AA}=0$ and $c_{\BB}=0$, then $\nS$ has order one and it holds
$\nL=c_{[\AA,\BB]}[\AA,\BB]+\Order(\tau^3)$. 
Because $c_{\AA\BB}$ is the coefficient of $\AA\BB$ in $\nL$ and thus in $c_{[\AA,\BB]}[\AA,\BB]$, it follows  $c_{[\AA,\BB]}=c_{\AA\BB}$. 
Therefore, if $c_{\AA}=0$, $c_{\BB}=0$, and $c_{\AA\BB}=0$, then $\nS$ has order two
and it holds 
$\nL=c_{[\AA,[\AA,\BB]]}[\AA,[\AA,\BB]]+c_{[[\AA,\BB],\BB]}[[\AA,\BB],\BB]+\Order(\tau^4)$.
Because $c_{\AA\AA\BB}$ is the coefficient of $\AA\AA\BB$ in $\nL$ and thus in
$c_{[\AA,[\AA,\BB]]}[\AA,[\AA,\BB]]+c_{[[\AA,\BB],\BB]}[[\AA,\BB],\BB]$ it follows
$c_{[\AA,[\AA,\BB]]}=c_{\AA\AA\BB}$. Similarly, $c_{[[\AA,\BB],\BB]}=c_{\AA\BB\BB}$.
Together, we have 5 order conditions
$c_{\AA}=0$, $c_{\BB}=0$, $c_{\AA\BB}=0$, $c_{\AA\AA\BB}=0$, $c_{\AA\BB\BB}=0$
for the 6 parameters $a_1, a_2, a_3, b_1, b_2, b_3$ which, if satisfied, ensure that
$\nS$ has order 3.\footnote{To choose a ``good'' solution from the one-dimensional
solution manifold of this system of equations,
one can minimize the
``local error measure''
$\sqrt{c_{\AA\AA\AA\BB}^2+c_{\AA\AA\BB\BB}^2+c_{\AA\BB\BB\BB}^2}$ made up from the
coefficients of all Lyndon words of length four in $\nL$, see 
\cite[Section~4]{part1}.}
A general version of this procedure for the generation of order conditions is provided
by our  Theorem~\ref{thm:order_conditions}.
A similar procedure was already proposed in \cite{part1}, see also
\cite{auzingeretal13c} and \cite[Section~7]{auzingeretal16b}. However, 
our Theorem~\ref{thm:order_conditions} is much more general and, furthermore, our
method for computing coefficients of Lyndon words (see Section~\ref{intro:coeff}) is
much more efficient than the one proposed in these references.

\subsection{Applications to Magnus-type methods}
\label{intro:magnus}
Our considerations about %(generalized) 
splitting methods
%for evolution equations of type (\ref{eq:split_evolution_eq}) 
can be transferred to
Magnus-type integrators for the numerical solution of
non-autonomous evolution equations (\ref{eq:non_auto_evolution_eq}).
Following \cite[Section~3]{alvfeh11} we expand $A(t)$ locally on an interval $[t_n, t_n+\tau]$ of length $\tau>0$ into a series of 
Legendre polynomials shifted to $[0, \tau]$,
\begin{equation}\label{eq:A_expansion}
A(t_n+t)=A_1\tilde{P}_0(t)+A_2\tilde{P}_1(t)+A_3\tilde{P}_2(t)+\dots,\quad t\in[0,\tau],
\end{equation}
where
\begin{equation}\label{eq:legendre_zeugs}
\tilde P_k(t) = \frac{1}{\tau}P_{k}\left(\frac{t}{\tau}\right),\quad
P_{k}(x)=(-1)^k\sum_{j=0}^k{k \choose j}{k+j \choose j}(-1)^j x^j.
\end{equation}
%Here the explicit expression for the Legendre polynomials $P_k(x)$ will be needed later.
The matrix-valued coefficients $A_1, A_2, A_3,\dots$ are defined as
\begin{equation}\label{eq:A_integral}
A_k%=(2k-1)\tau\int_0^\tau A(t_n+t)\tilde{P}_{k-1}(t)\,\mathrm{d}t
=(2k-1)\tau\int_0^1 P_{k-1}(x)A(t_n+\tau x)\,\mathrm{d}x,
\end{equation}
they depend on both $t_n$ and $\tau$, and it holds
\begin{equation} \label{eq:legendre_coeffs}
A_k=\Order(\tau^k).
\end{equation}
In practice they are approximately calculated using Gaussian quadrature, see \cite[Section~7]{alvfeh11}.\footnote{The expansion into Legendre polynomials
proves to be very convenient. For some theoretical considerations, however,  other expansions may be
more suitable, e.g., Taylor expansion around $t_n$ or around 
the midpoint $t_n+\tfrac{1}{2}\tau$ of $[t_n,t_n+\tau]$, 
where respectively $P_k(x)=x^k$ or $P_k(x)=(x-\tfrac{1}{2})^k$ instead 
of (\ref{eq:legendre_zeugs}). Our considerations carry over to these cases
(with  modified representations of the Magnus series $\Omega$, of course),
provided that the expansion coefficients satisfy (\ref{eq:legendre_coeffs}).
}
One step of step-size $\tau$ 
of a Magnus-type integrator of order $p$
can be written as
\begin{equation}\label{eq:non_auto_step}
t_n\mapsto t_{n+1}=t_n+\tau,\quad u_n\mapsto u_{n+1}=\nS(\tau, t_n)u_n,
\end{equation} 
where $\nS(\tau, t_n)$ is an approximation of 
the exact local solution operator, 
\begin{equation*}
\nS(\tau, t_n)=\nE(\tau, t_n) +\Order(\tau^{p+1})= \ee^{\Omega}+\Order(\tau^{p+1}),
\end{equation*}
where $\Omega$ denotes the {\em Magnus} series
\begin{eqnarray}
	\Omega &=&  A_1-\tfrac{1}{6}[A_1,A_2]+\tfrac{1}{60}[A_1,[A_1,A_3]]-\tfrac{1}{60}[A_2,[A_1,A_2]]\nonumber\\
	&&+\tfrac{1}{360}[A_1,[A_1,[A_1,A_2]]]-\tfrac{1}{30}[A_2,A_3]%-\frac{1}{70}[A3,A4] A+ 
	+\dots
	\label{eq:magnus_series}
\end{eqnarray}	 
in terms of the Legendre expansion coefficients $A_k$, see \cite[Section~3.2]{alvfeh11}.
We give some prototypical examples of Magnus-type integrators:
\begin{itemize} 
\item
Classical fourth order Magnus integrator, obtained by truncating the Magnus series $\Omega$ in $\ee^\Omega$:
\begin{equation*}
\nS(\tau, t_n)=\ee^{A_1-\frac{1}{6}[A_1,A_2]}.
\end{equation*}
\item
Fourth order commutator-free integrator \cite[eq.~(38)]{alvfeh11}:
\begin{equation}\label{eq:comfree4}
\nS(\tau, t_n)=\ee^{\frac{1}{2}A_1+\frac{1}{3}A_2}\,\ee^{\frac{1}{2}A_1-\frac{1}{3}A_2}.
\end{equation}
\item
The following %``unconventional'' 
scheme of order six %for the solution of (\ref{eq:non_auto_evolution_eq}) 
involving  one commutator in the middle exponential was
 proposed in \cite{SergioFernandoMPaper2}:
\begin{align}
\nS(\tau,t_n)=&\
\ee^{f_{11}A_1-f_{12}A_2+f_{13}A_3}\,
\ee^{f_{21}A_1-f_{22}A_2+f_{23}A_3}%\nonumber\\
\,\,\ee^{[g_1A_1+g_3A_3,A_2]}\nonumber\\
&\times\,\ee^{f_{21}A_1+f_{22}A_2+f_{23}A_3}\,
\ee^{f_{11}A_1+f_{12}A_2+f_{13}A_3}\label{eq:unconv_scheme}
\end{align}
with coefficients
\begin{align}
(f_{jk}) \doteq\ & \left({\fontsize{8.5}{9.}\selectfont\begin{array}{rrr}
 0.166598694406302053& 
-0.150420414495444186& 
 0.119990212792817809\\ 
 0.333401305593697947& 
-0.127503033859797053& 
-0.119990212792817809\end{array}}\right),\nonumber\\
 g_1\doteq\ & 0.001203581117795540,\quad g_3\doteq -0.000014760374925774. 
 \label{eq:coeff_unconv}
\end{align}
\end{itemize}
For such methods $\nS$
the leading term in the local error $\nL=\nS-\nE$ is again a
homogeneous Lie element of order $\Order(\tau^{p+1})$, 
our general
Theorem~\ref{thm:I} covers this case, too.
To calculate a
representation of the leading local error term  in the Lyndon basis,
analogously as in the case of splitting methods,
one  has to compute the coefficients of all 
Lyndon words of order $\Order(\tau^{p+1})$ over the alphabet $\{\AA_1,\AA_2,\dots\}$.
The coefficient of a Lyndon word $w$ in $\nL=\nS-\nE$ is the difference of the
coefficient of $w$ in $\nS$ to be computed using the algorithm of Section~\ref{intro:coeff} and the coefficient of $w$ in $\nE=\ee^{\Omega}$, 
which in principle can also be computed by this algorithm, provided the terms of the Magnus series $\Omega$ in (\ref{eq:magnus_series}) are available up to order $\Order(\tau^{p+1})$.
However, our Theorem~\ref{eq:coeff_rhs_magnus_type} provides an explicit formula for
this coefficient, so that explicit knowledge of the Magnus series  is not
necessary. 
These considerations again lead to a procedure for the generation of order conditions for Magnus-type integrators, similar to the one for splitting methods, see Theorem~\ref{thm:order_conditions}.

\section{Theoretical considerations}\label{sct:theory}
\subsection{Algebraic setting} 
We consider the free Lie algebra  $\gg=[\mathbb{C}\langle\nA\rangle]$ over an appropriate set of generators $\nA$.\footnote{Here 
``free'' means that we consider the generic case. In particular, we do not assume that there hold any relations between Lie elements except those which follow from the axioms defining a Lie algebra. 
}
A grading function specified by its values 
on $\nA$ will  turn $\gg$ into a graded Lie algebra, see \cite{MuntheKaas957}.
\begin{itemize}
\item
For the study of splitting methods 
we set  $\nA=\{\AA, \BB\}$ were $\AA, \BB$ represent
$\tau A, \tau B$, respectively. Corresponding to
$\AA \simeq  \tau A =\Order(\tau)$ and $\BB \simeq \tau B =\Order(\tau)$ we define
\begin{equation}\label{eq:grade_gen_split}
  \grade(\AA)=\grade(\BB)=1.
\end{equation}
\item
For Magnus-type integrators we set  $\nA=\{\AA_1\dots, \AA_K\}$
where  $K\geq 2$ depends on the particular scheme, and the $\AA_k$ represent the Legendre coefficients $A_k$ from (\ref{eq:A_expansion}). Corresponding to  (\ref{eq:legendre_coeffs})
we define
\begin{equation}\label{eq:grade_gen_magnus_type}
 \grade(\AA_k) = k.
\end{equation}
\end{itemize}
We call iterated commutators $\in\gg$ with single generators $\in\nA$ in their slots, {\em pure} elements of $\gg$, or more formally: (i) all elements of the generating set $\nA$ are pure;
(ii) if $X,Y\in\gg$ are pure, then the commutator $[X,Y]$ is pure;
(iii)  elements of $\gg$ which can not be constructed by (i) or (ii) are not pure.

We define a grading function recursively for pure elements by (\ref{eq:grade_gen_split})
or (\ref{eq:grade_gen_magnus_type}) and
$$
\grade([X,Y]) = \grade(X)+\grade(Y),\quad X,Y\ \mbox{pure}.
$$
With
$$
\gg_k = \mathrm{span}\{X\in\gg:\ X\ \mbox{pure and}\ \grade(X)=k \}
$$
$\gg$ becomes a {\em graded} Lie algebra
$$
\gg=\bigoplus_{k=1}^\infty \gg_k.
$$
Each $\Phi\in\gg$, $\Phi\neq 0$ has a unique representation as a {\em finite} sum $\Phi=X_1+\ldots+X_q$ with
$X_k\in\gg_k$, $X_q\neq 0$.
Elements of $\gg_k$ are called homogeneous Lie elements of grade $k$.
Because there are only finitely many pure elements of fixed grade $k$, each
$\gg_k$ is finite-dimensional.

The universal enveloping algebra of $\gg$ is given by $\mathbb{C}\langle\nA\rangle\supset\gg$,
the algebra of polynomials in the non-commuting variables $\nA$ or, equivalently, the free associative
algebra generated by $\nA$. 
Values of exponentials of Lie elements $\in\gg$ are in a natural way  elements of
$\mathbb{C}\langle\langle\nA\rangle\rangle\supset\mathbb{C}\langle\nA\rangle$,
the algebra of formal power series in the non-commuting variables $\nA$.

We extend the grading function
(\ref{eq:grade_gen_split}),
 (\ref{eq:grade_gen_magnus_type})
 to words $w=w_1\cdots w_{\ell(w)}\in\nA^*\subset\mathbb{C}\langle\nA\rangle$ 
 with
$w_j\in\nA$, $j=1,\dots,\ell(w)=\mbox{length}(w)$,
$$
\grade(w) = \sum_{j=1}^{\ell(w)} \grade(w_j).
$$
Analogously as before, with
$$
V_k = \mathrm{span}\{w\in\nA^*:\ \grade(w)=k \}\subset\mathbb{C}\langle\nA\rangle,
$$
$\mathbb{C}\langle\nA\rangle$ becomes a {\em graded}  algebra
$$
\mathbb{C}\langle\nA\rangle=\bigoplus_{k=0}^\infty V_k.
$$
Here the direct sum starts with $V_0=\mathrm{span}\{\Id\}$, where $\Id$ is the
empty word with $\grade(\Id)=0$, which serves as the multiplicative identity of the algebra $\mathbb{C}\langle\nA\rangle$.

For $q\geq 1$ the subspace
$$\widetilde\nR_q=
  \mathrm{span}\{w\in\nA^* :\  \grade(w)\geq q \}
 = \bigoplus_{k=q}^\infty V_k
$$
is an ideal of $\mathbb{C}\langle\nA\rangle$. We define
$$
\nR_q = \text{the ideal of } \mathbb{C}\langle\langle\nA\rangle\rangle \text{ generated by } \widetilde\nR_q,
$$
which consists of all  series with terms in $\widetilde\nR_q$.
With  $\nR_q$ we have a substitute in  $\mathbb{C}\langle\langle\nA\rangle\rangle$
for computations up to order $q$,
mnemonically
$$
\nR_q \simeq \Order(\tau^q).
$$

\begin{lemma}\label{lemma:0}
Let $\Phi=X_1+\dots +X_q$, $X_k\in\gg_k$, 
$\Psi=Y_1+\dots +Y_q$, $Y_k\in\gg_k$, and $R_1,R_2,R_3\in\nR_{q+1}$.
If 
\begin{equation}\label{eq:lemma0vor}
\ee^{\Phi+R_1}=\ee^{\Psi+R_2}+R_3,
\end{equation}
then
$$ \Phi=\Psi.$$
\end{lemma}
\begin{proof} Expanding and rearranging
(\ref{eq:lemma0vor})
yields
$$0=\Phi-\Psi+\tfrac{1}{2}(\Phi^2-\Psi^2)+\ldots+\tfrac{1}{q!}(\Phi^q-\Psi^q)+R\in\mathbb{C}\langle\langle\nA\rangle\rangle=
\bigoplus_{k=0}^q V_k\oplus\nR_{q+1}$$ for some $R\in\nR_{q+1}$.
Here the component in $V_1$ is $X_1-Y_1$, thus $X_1=Y_1$.
Then using $X_1=Y_1$ it is easy to see that the component in
$V_2$ is $X_2-Y_2$, thus $X_2=Y_2$. 
Continuing in this way we obtain $X_3=Y_3$, \dots, $X_q=Y_q$. \qed
\end{proof}
\begin{lemma}\label{lemma:I}
Let $\Phi_1,\dots,\Phi_J\in\gg$. %=[\mathbb{C}\langle\nA\rangle]$. 
Then there exists a sequence $X_k\in\gg_k$, $k=1,2,\dots$ such that
for all $q\geq 1$
\begin{equation*}
\ee^{\Phi_J}\cdots\ee^{\Phi_1}=\ee^{\Psi_q} + R_{q+1},\quad
\Psi_q=X_1+\ldots+X_q, \ R_{q+1}\in \nR_{q+1}.
\end{equation*}
\end{lemma}
\begin{proof}
For $J=1$ the statement of the lemma holds trivially.
 
Next we give the proof for the case $J=2$. 
By the Baker-Campbell-Hausdorff (BCH) formula (see, e.g., \cite[Section~2.8]{BLANES2009151})
there exist $X_k\in\gg_k$, $k=1,2,\dots$ with
$$
\ee^{\Phi_2}\ee^{\Phi_1} = \ee^{\hat\Psi},\quad\hat\Psi=X_1+X_2+\dots.
$$
With $\Psi_q=X_1+\dots+ X_q$  it holds $\hat\Psi=\Psi_q+\hat{R}_{q+1}$ with $\hat{R}_{q+1}\in\nR_{q+1}$.
It follows
\begin{align*}
\ee^{\hat\Psi} &=\mathrm{Id}+(\Psi_q+\hat{R}_{q+1})+\tfrac{1}{2}(\Psi_q+\hat{R}_{q+1})^2+\ldots\\
&=\Id+\Psi_q+\tfrac{1}{2}\Psi_q^2+\ldots+R_{q+1}
=\ee^{\Psi_q}+R_{q+1}
\end{align*}
for some  $R_{q+1}\in\nR_{q+1}$.

By induction we obtain $\ee^{\Phi_J}\cdots\ee^{\Phi_1}=\ee^{\Psi_q} + R_{q+1}$ for some
$\Psi_q=X_1^{(q)}+\ldots+X_q^{(q)}$, $X_k^{(q)}\in\gg_k$, and 
$R_{q+1}\in \nR_{q+1}.$  Using Lemma~\ref{lemma:0} we conclude that the $X_k^{(q)}$ do not depend on $q$, i.e.,
$X_1^{(1)}=X_1^{(2)}=\ldots=X_q^{(q)}$ for all $q\geq 1$,
 and the statement of the lemma follows. \qed
\end{proof}

The following theorem states that
if an exponential $\ee^{\Omega}$ of a Lie element is approximated by a product $\ee^{\Phi_J}\cdots\ee^{\Phi_1}$ of exponentials of Lie elements, 
then
the leading error term $\Theta$ is a homogeneous Lie element of some grade $q\geq 1$. 
This situation occurs if $\ee^\Omega$ represents the exact solution operator of (\ref{eq:split_evolution_eq}) or
(\ref{eq:non_auto_evolution_eq}), i.e., if $\Omega=\AA+\BB$ or $\Omega=$ (truncated) Magnus
series (\ref{eq:magnus_series}), and the scheme $\ee^{\Phi_J}\cdots\ee^{\Phi_1}$ represents
respectively a generalized
splitting method or a Magnus-type integrator.

\begin{theorem}\label{thm:I}
Let $\Phi_1,\dots,\Phi_J,\Omega\in\gg$. %=[\mathbb{C}\langle\nA\rangle]$. 
Then for some $q\geq 1$ there
exist $\Theta\in\gg_q$, $\Theta\neq 0$ and $R\in\nR_{q+1}$ such that
\begin{equation}\label{eq:basic_formula}
\ee^{\Phi_J}\cdots\ee^{\Phi_1}=\ee^{\Omega} + \Theta + R.
\end{equation}
\end{theorem}
\begin{proof}
By Lemma~\ref{lemma:I} there exist $X_k\in\gg_k$, $k=1,2,\dots$ such that for all $r\geq 1$
\begin{equation*}
\ee^{\Phi_J}\cdots\ee^{\Phi_1}\ee^{-\Omega} = \ee^{X_1+\ldots+X_r} + R_{r+1}
\end{equation*}
with $R_{r+1}\in\nR_{r+1}$.
The case that $X_k=0$ for all $k$ is only possible if 
$\ee^{\Phi_J}\cdots\ee^{\Phi_1}\ee^{-\Omega}=\Id$ and thus
$\ee^{\Phi_J}\cdots\ee^{\Phi_1}=\ee^{\Omega}$, from which the statement of the theorem
follows for $q\geq 1$ arbitrary and $\Theta=0$, $R=0$.

So let us assume that $X_k\neq 0$ for some $k\geq 1$.
Set $q=\min\{k:\ X_k\neq 0\}$. For $r=q$ it holds
\begin{equation*}
\ee^{\Phi_J}\cdots\ee^{\Phi_1}\ee^{-\Omega} = \ee^{X_q} + R_{q+1}.
\end{equation*}
Consequently,
\begin{align*}
\ee^{\Phi_J}\cdots\ee^{\Phi_1}\ee^{-\Omega}-\Id&= \ee^{X_q} + R_{q+1}- \Id \\
&= \Id+X_q+\tfrac{1}{2}X_q^2+\tfrac{1}{6}X_q^3+\ldots+R_{q+1}-\Id\\
&=X_q+R_{q+1}',
\end{align*}
where $R_{q+1}'=\tfrac{1}{2}X_q^2+\tfrac{1}{6}X_q^3+\ldots+R_{q+1}\in\nR_{q+1}$.
With $Y_k\in\gg_k$ such that $\Omega=Y_1+\ldots+Y_s$ it follows
\begin{align*}
\ee^{\Phi_J}\cdots\ee^{\Phi_1}-\ee^{\Omega}  &= (\ee^{\Phi_J}\cdots\ee^{\Phi_1}\ee^{-\Omega}-\Id)\ee^{\Omega}
=(X_q+R_{q+1}')\ee^{\Omega}\\
&=X_q\ee^{Y_1+\ldots+Y_s}+R_{q+1}'\ee^{\Omega}\\
&=X_q+X_q(Y_1+\ldots+Y_s)+\tfrac{1}{2}X_q(Y_1+\ldots+Y_s)^2+\ldots+R_{q+1}'\ee^{\Omega}\\
&=\Theta+R,
\end{align*}
where $\Theta=X_q\in\gg_q$ and $R=X_q(Y_1+\ldots+Y_s)+\tfrac{1}{2}X_q(Y_1+\ldots+Y_s)^2+\ldots+R_{q+1}'\ee^{\Omega}\in \nR_{q+1}$. \qed
\end{proof}

\subsection{Symmetry}\label{subsct:symmetry}
One-step methods like (generalized) splitting and Magnus-type methods 
(cf.~(\ref{eq:step_split}), (\ref{eq:non_auto_step})) are called
self-adjoint or symmetric if they satisfy
$$\nS(-\tau,t_n+\tau)\nS(\tau,t_n)=\Id.$$
We give a purely formal definition of this property for 
products of exponentials $\ee^{\Phi_J}\cdots\ee^{\Phi_1}$ representing
such methods. To this end,
we define operations $(\cdot)\widehat{\ }$, $(\cdot)\widetilde{\ }$,
$(\cdot)\widetilde{\widehat{\ }}$
 on 
Lie elements $\Phi=\sum_{k=1}^{K}X_k$, $X_k\in\gg_k$ by
\begin{equation}\label{eq:tilde_hat_operations}
\widetilde{\Phi} = \sum_{k=1}^{K}(-1)^{k+1}X_k,\quad
\widehat{\Phi} = -\Phi,\quad
\widehat{\widetilde{\Phi}}= -\widetilde{\Phi},
\end{equation}
and extend these operations to products of exponentials,
\begin{align*}
&\big(\ee^{\Phi_J}\cdots\ee^{\Phi_1}\big)\widehat{\ } =
\ee^{-\Phi_J}\cdots\ee^{-\Phi_1},\quad
\big(\ee^{\Phi_J}\cdots\ee^{\Phi_1}\big)\widetilde{\ } =
\ee^{\widetilde{\Phi}_J}\cdots\ee^{\widetilde{\Phi}_1},\\
&\big(\ee^{\Phi_J}\cdots\ee^{\Phi_1}\big)\widehat{\widetilde{\ }} =
\ee^{-\widetilde{\Phi}_J}\cdots\ee^{-\widetilde{\Phi}_1}.
\end{align*}
The adjoint of a product of exponentials $\ee^{\Phi_J}\cdots\ee^{\Phi_1}$ is defined by
\begin{equation*}
\big(\ee^{\Phi_J}\cdots\ee^{\Phi_1}\big)^* = \ee^{\widetilde{\Phi}_1}\cdots\ee^{\widetilde{\Phi}_J}
\end{equation*}
(note the order of the exponentials),
and we say that %a product of exponentials 
$\ee^{\Phi_J}\cdots\ee^{\Phi_1}$ is self-adjoint or symmetric if
\begin{equation*}%\label{eq:selfadjoint}
\big(\ee^{\Phi_J}\cdots\ee^{\Phi_1}\big)^* = \ee^{\Phi_J}\cdots\ee^{\Phi_1}.
\end{equation*}
A single exponential $\ee^\Phi$ is self-adjoint, $(\ee^\Phi)^*=\ee^\Phi$,
if and only if $\widetilde{\Phi}=\Phi$,
if and only if $\Phi$ is a sum of homogeneous Lie elements of odd grade,
\begin{equation*}
\Phi=X_1+X_3+\dots+X_K,\quad X_k\in\gg_k,\ K \mbox{ odd}.
\end{equation*}
It follows that $\ee^\Omega$ representing
the exact solution operator for (\ref{eq:split_evolution_eq}) or (\ref{eq:non_auto_evolution_eq})
is self-adjoint, where $\Omega=\AA+\BB$ or $\Omega=$ (truncated) Magnus series (\ref{eq:magnus_series}).

The following theorem improves Theorem~\ref{thm:I} for the
self-adjoint case.
\begin{theorem}\label{thm:I_selfadjoint}
Let $\Phi_1,\dots,\Phi_J,\Omega\in\gg$ %=[\mathbb{C}\langle\nA\rangle]$
with $\nS=\ee^{\Phi_J}\cdots\ee^{\Phi_1}$ and $\ee^\Omega$
self-adjoint. 
Then for some {\em odd} $q\geq 1$ there
exist $\Theta\in\gg_q$, $\Theta\neq 0$ and $R\in\nR_{q+1}$ such that
\begin{equation*}%\label{eq:basic_formula_sa}
\ee^{\Phi_J}\cdots\ee^{\Phi_1}=\ee^{\Omega} + \Theta + R.
\end{equation*}
\end{theorem}
\begin{proof}
By Theorem~\ref{thm:I},    
$\nS=\ee^{\Omega}+\Theta+R$
for some $\Theta\in\gg_q$, 
$\Theta\neq 0$, $R\in\nR_{q+1}$.
We prove that $q$ has to be odd under the given assumptions.
On one hand,
\begin{equation*}
\widehat{\widetilde{\!\nS\ }}\!\!\!\nS
=\widehat{\widetilde{\!\nS\ }}\!\!\!\nS^*
%=\big(\ee^{\widetilde{\Phi}_J}\cdots\ee^{\widetilde{\Phi}_1}\big)\widehat{\ }\,\ee^{\Phi_J}\cdots\ee^{\Phi_1}
=\ee^{-\widetilde{\Phi}_J}\cdots\ee^{-\widetilde{\Phi}_1}\ee^{\widetilde{\Phi}_1}\cdots\ee^{\widetilde{\Phi}_J}=\Id,
%=(\ee^{\Phi_1}\cdots\ee^{\Phi_J})\widehat{\ }\,\ee^{\Phi_J}\cdots\ee^{\Phi_1}\\
%&=\ee^{-\Phi_1}\cdots\ee^{-\Phi_J}\ee^{\Phi_J}\cdots\ee^{\Phi_1}
%=\Id.
\end{equation*}
On the other hand,
substituting the generators 
$\AA\mapsto\widehat{\widetilde{\AA}}$,
$\BB\mapsto\widehat{\widetilde{\BB}}$, or 
$\AA_k\mapsto\widehat{\widetilde{\AA}}_k, k=1,2,\dots$ in
$\nS=\ee^{\Phi_J}\cdots\ee^{\Phi_1}=\ee^{\Omega}+\Theta+R$
and using $\widetilde{\Omega}=\Omega$
we obtain
$\widehat{\widetilde{\!\nS\ }}\!\!
= \ee^{-\Omega}-\widetilde{\Theta}
-\widetilde{R}$, $\widetilde{R}\in\nR_{q+1}$, and thus
\begin{align*}
\widehat{\widetilde{\!\nS\ }}\!\!\!\nS
&=(\ee^{-\Omega}-\widetilde{\Theta}
-\widetilde{R})(\ee^{\Omega}+\Theta+R)\quad&& \\
&=\Id+\ee^{-\Omega}\Theta-\ee^{\Omega}\widetilde{\Theta} + R' && (\mbox{for some } R'\in \nR_{q+1})\\
&=\Id+(\Id-\Omega+\dots)\Theta-
(\Id+\Omega+\dots)\widetilde{\Theta}+R'\\
&=\Id+\Theta-\widetilde{\Theta}+R''
 && (\mbox{for some } R''\in \nR_{q+1}).
\end{align*}
It follows $\Theta=\widetilde{\Theta}$ for $\Theta\in\gg_k$ 
and thus $q$ odd by (\ref{eq:tilde_hat_operations}). \qed
\end{proof}

\subsection{Lyndon words, Lyndon basis, order conditions, local error measure}

Our next goal is to  determine explicitly the leading error term $\Theta\in\gg_q$ in (\ref{eq:basic_formula}) whose existence is guaranteed by Theorem~\ref{thm:I}.
Thus, given a basis $\nB_q$ of $\gg_q$ we want to calculate the coefficients
$c_b=\coeff_{\nB_q}(b,\Theta)$ in the representation
\begin{equation}\label{eq:lin_comb_Theta}
\Theta =\sum_{b\in\nB_q}c_b b.
\end{equation}
A convenient choice for the basis $\nB_q=\nB_q^{\nA}$ of $\gg_q=\gg_q^\nA$
is the Lyndon basis whose elements
correspond uniquely to Lyndon words over
the alphabets $\nA=\{\AA,\BB\}$ or $\nA=\{\AA_1,\dots,\AA_K\}$.
Considering the lexicographical order `$<$' on words (with $\AA<\BB$ and $\AA_1<\ldots<\AA_K$),
a word $w=w_1\cdots w_{n}\in\nA^*$ is a Lyndon word if it is strictly less than any of its proper right
factors $w_i\cdots w_{n}$, $i=2,\dots,n$.
For a standard reference for Lyndon words, see \cite{lothaire_1997}.
%; an efficient algorithm for listing Lyndon words in lexicographical order is provided %by Duval \cite{DUVAL1988255}. 
An algorithm for the efficient computation of all  
Lyndon words of given {\em length} in
lexicographical order is given in
\cite[Algorithm~2.1]{CATTELL2000267}. 
Since $\grade(w)=\mathrm{length}(w)$ for words $w$ over the
alphabet $\nA=\{\AA,\BB\}$, this algorithm computes 
$$\nW_q^{\nA}=\{
\text{Lyndon words of grade } q\text { over the alphabet } \nA\}$$
directly for $\nA=\{\AA,\BB\}$.
For $\nW_q^{\{\AA_1,\dots,\AA_q\}}$ %for $\nA=\{\AA_1,\dots,\AA_q\}$ 
we can then use
\begin{equation*}
\AA\underbrace{\BB\cdots\BB}_{d_1-1}
\AA\underbrace{\BB\cdots\BB}_{d_2-1}\cdots
\AA\underbrace{\BB\cdots\BB}_{d_\ell-1}
\mapsto \AA_{d_1}\AA_{d_2}\cdots\AA_{d_\ell},
\end{equation*}
which defines a bijective and lexicographic order preserving 
mapping $\nW_q^{\{\AA,\BB\}}\to\nW_q^{\{\AA_1,\dots,\AA_q\}}$.

\begin{table}[t!]
\begin{center}
\begin{tabular}{|c|l|l|}
\hline
$q$ & Lyndon words & Lyndon basis %& rightnormed basis 
\\
\hline
\hline
1   & $\AA$, $\BB$ 
    & $\AA$, $\BB$ 
%    & $\AA$, $\BB$  
\\
\hline 
2	& $\AA\BB$ 
    & $[\AA,\BB]$  
%    & $[\BB,\AA]$  
\\    
\hline
3   & $\AA\AA\BB$, $\AA\BB\BB$ 
	& $[\AA,[\AA,\BB]]$,\ $[[\AA,\BB],\BB]$ 
%	&  $[\AA,[\BB,\AA]]$,  $[\BB,[\BB,\AA]]$ 
\\
\hline
4   &  $\AA\AA\AA\BB$,  $\AA\AA\BB\BB$, $\AA\BB\BB\BB$
    &  $[\AA,[\AA,[\AA,\BB]]]$,\ $[\AA,[[\AA,\BB],\BB]]$,\ $[[[\AA,\BB],\BB],\BB]$ 
%    & $[\AA,[\AA,[\BB,\AA]]]$, $[\BB,[\AA,[\BB,\AA]]]$, $[\BB,[\BB,[\BB,\AA]]]$
\\    
\hline
5   &$\AA\AA\AA\AA\BB$, $\AA\AA\AA\BB\BB$, $\AA\AA\BB\AA\BB$,
	&$[\AA,[\AA,[\AA,[\AA,\BB]]]]$,\ $[\AA,[\AA,[[\AA,\BB],\BB]]]$,\ $[[\AA,[\AA,\BB]],[\AA,\BB]]$,
%    & $[\AA,[\AA,[\AA,[\BB,\AA]]]]$, $[\BB,[\AA,[\AA,[\BB,\AA]]]]$,$[\AA,[\BB,[\AA,[\BB,\AA]]]]$, 
    \\ 
    &$\AA\AA\BB\BB\BB$, $\AA\BB\AA\BB\BB$, $\AA\BB\BB\BB\BB$
    &$[\AA,[[[\AA,\BB],\BB],\BB]]$,\ $[[\AA,\BB],[[\AA,\BB],\BB]]$,\ $[[[[\AA,\BB],\BB],\BB],\BB]$
%    &$[\BB,[\BB,[\AA,[\BB,\AA]]]]$,$[\AA,[\BB,[\BB,[\BB,\AA]]]]$, $[\BB,[\BB,[\BB,[\BB,\AA]]]]$
    \\ 
\hline
\end{tabular}
\caption{Lyndon words $\nW_q^\nA$ of grade $q$ and Lyndon basis $\nB_q^\nA$ of $\gg_q$ for generators $\nA=\{\AA,\BB\}$ with grading (\ref{eq:grade_gen_split}).
\label{tab:lyndon_a_b}
}
\end{center}
\end{table}

\begin{table}[t!]
\begin{center}
\begin{tabular}{|c|l|l|}
\hline
$q$ &Lyndon words & Lyndon basis %& rightnormed basis\!\!\! 
\\
\hline
\hline
1 & $\AA_1$ & $\AA_1$ %& $\AA_1$ 
\\
\hline
2 &  $\AA_2$ & $\AA_2$ %& $\AA_2$ 
\\
\hline
3 & $\AA_1\AA_2$, $\AA_3$ 
  & $[\AA_1,\AA_2]$, $\AA_3$ 
%  & $[\AA_2,\AA_1]$, $\AA_3$ 
\\
\hline
4 & $\AA_1\AA_1\AA_2$, $\AA_1\AA_3$, $\AA_4$
  & $[\AA_1,[\AA_1,\AA_2]]$, \ $[\AA_1,\AA_3]$, \ $\AA_4$
%  &$[\AA_1,[\AA_2,\AA_1]]$, $[\AA_3,\AA_1]$, $\AA_4$ 
\\
\hline
5 & $\AA_1\AA_1\AA_1\AA_2$, $\AA_1\AA_1\AA_3$, $\AA_1\AA_2\AA_2$,
  & $[\AA_1,[\AA_1,[\AA_1,\AA_2]]]$, \ $[\AA_1,[\AA_1,\AA_3]]$, \ $[[\AA_1,\AA_2],\AA_2]$,
%  & $[\AA_1,[\AA_1,[\AA_2,\AA_1]]]$, $[\AA_1,[\AA_3,\AA_1]]$, $[\AA_2,[\AA_2,\AA_1]]$,  
 \\
 & $\AA_1\AA_4$, $\AA_2\AA_3$, $\AA_5$
 & $[\AA_1,\AA_4]$, \ $[\AA_2,\AA_3]$, \ $\AA_5$
%  $[\AA_4,\AA_1]$,  &$[\AA_3,\AA_2]$, $\AA_5$
  \\          
\hline
\end{tabular}
\caption{Lyndon words $\nW_q^\nA$ of grade $q$ and Lyndon basis $\nB_q^\nA$ of $\gg_q$ for generators $\nA=\{\AA_1,\dots,\AA_q\}$ with grading (\ref{eq:grade_gen_magnus_type}).
\label{tab:lyndon_ak}
}
\end{center}
\end{table}

Let $w$ be a Lyndon word of length at least 2, and let $u,v$ be words such that $w=uv$ and $v$ 
is the longest  Lyndon word appearing as a proper right factor of $w$. Then $u$ is also a
Lyndon word, and $w=uv$ is called the right standard factorization of $w$.
We define the standard bracketing of Lyndon words recursively by
\begin{align*}
   &\beta(g)=g, &&\mbox{if}\ g\in\nA, \\
   &\beta(w)=[\beta(u),\beta(v)], &&\mbox{if %$\mathrm{length}(w)\geq 2$ and 
   $w=uv$ is the right standard factorization of $w$}.
\end{align*}
Then 
the Lyndon basis of $\gg_q=\gg_q^{\nA}$ is given by
$$
\nB_q^{\nA}=\{\beta(w): w\in\nW_q^{\nA}\}.
$$
For an efficient algorithm for the computation of $\beta(w)$ for all
Lyndon words $w$ of given {\em length} 
%in lexicographical order 
see \cite{SAWADA200321}. This algorithm computes $\nB_q^{\nA}$ directly for 
$\nA=\{\AA,\BB\}$, and then $\nB_q^{\{\AA_1,\dots,\AA_q\}}$ can
be obtained by performing all possible
substitutions
\begin{equation*}
   [[[\ldots[\AA,\underbrace{\BB],\ldots],\BB],\BB}_{d-1}] \mapsto \AA_d
\end{equation*}
of sub-commutators in the elements of $\nB_q^{\{\AA,\BB\}}$.

Consider 
the matrix
$$T_q=T_{\nB_{q}\to\nW_{q}} = \big(\coeff(v,\beta(w))\big)_{v\in\nW_q,w\in\nW_q},$$
where it is assumed that the Lyndon words $v,w$ are traversed in lexicographical order.
For example, for $\nA=\{\AA,\BB\}$ (cf. Table~\ref{tab:lyndon_a_b}),\footnote{A
detailed calculation leading to $T_5$ for $\nA=\{\AA,\BB\}$ can be found in 
\cite[Section~2.1]{part1}. Generally, the entries of these matrices can be calculated by applying the algorithm (\ref{ex:first_component})--(\ref{eq:expY}) from 
Section~\ref{subsct:homos} below.}
\begin{equation}\label{eq:T_splitting}
T_1=I_2,\ \ T_2=I_1,\ \ T_3=I_2,\ \ T_4=I_3,\ \
T_5=\left(
\begin{array}{rrrrrr}
1 \\
& 1 \\
& -2 & 1 \\
& & & 1 \\
& & & -3 & 1\\
& & & & & 1
\end{array}
\right),
\end{equation}
and for $\nA=\{\AA_1,\dots,\AA_q\}$ (cf. Table~\ref{tab:lyndon_ak}),
\begin{equation*}
T_1=I_1,\ \ T_2=I_1, \ \ T_3=I_2, \ \ T_4=I_3,\ \ T_5=I_6,
\end{equation*}
where $I_k$ denotes the identity matrix in $\mathbb{C}^{k\times k}$.
The situation shown in these examples occurs also in the general case,
the matrices $T_q$ are always lower triangular  with unit diagonal,
see \cite[Section~2.1]{part1}.
In particular, $T_q$ is invertible, the inverse $T_q^{-1}$ has integer entries (this follows from $\det T_q=1$), and the $c_b$ in (\ref{eq:lin_comb_Theta})
can be calculated via
\begin{equation}
\begin{split}
&\mathbf{c}_{\nB_q}=T_q^{-1}\cdot \mathbf{c}_{\nW_q},\\
&\mathbf{c}_{\nB_q} = (c_b) =\big(\coeff_{\nB_q}(b,\Theta)\big)_{b\in \nB_q},\quad
\mathbf{c}_{\nW_q} = (c_w) =\big(\coeff(w,\Theta)\big)_{w\in \nW_q}.
\end{split}
\label{eq:coeff_vectors}
\end{equation}
We have thus reduced the computation of coefficients $\coeff_{\nB_q}(b,\Theta)$ of commutators $b\in\mathcal{B}_q$ to the computation of coefficients
\begin{equation}\label{eq:coeff_w_theta}
\coeff(w,\Theta) = \coeff(w,\ee^{\Phi_J}\cdots\ee^{\Phi_1}) -\coeff(w,\ee^{\Omega})
\end{equation}
of words $w\in\mathcal{W}_q$.
%\subsection{Order conditions}

The following theorem introduces order conditions, which, if satisfied,
determine the value of $q$ in Theorem~\ref{thm:I}, and thus the
order of the scheme $\ee^{\Phi_J}\cdots\ee^{\Phi_1}$ as an approximation of 
$\ee^{\Omega}$.
\begin{theorem}[Order conditions]\label{thm:order_conditions}Let $\Phi_1,\dots,\Phi_J,\Omega\in\gg$ and
$p\geq 1$.
If the order conditions
\begin{equation}\label{eq:order_conditions}
\coeff(w,\ee^{\Phi_J}\cdots\ee^{\Phi_1}-\ee^{\Omega})=0,\quad 
%\mbox{for all Lyndon words }
w\in\bigcup_{q=1}^p\nW_q
%\mbox{ of grade } q\leq p
\end{equation}
are satisfied
for all Lyndon words of grade $q\leq p$,
then the scheme $\ee^{\Phi_J}\cdots\ee^{\Phi_1}$ has order $p$ as an approximation of $\ee^{\Omega}$, 
\begin{equation}\label{eq:order_stmnt}
\ee^{\Phi_J}\cdots\ee^{\Phi_1}-\ee^\Omega\in\nR_{p+1},
\end{equation}
or more precisely, applying Theorem~\ref{thm:I},
\begin{equation} \label{eq:leading_error_term}
\ee^{\Phi_J}\cdots\ee^{\Phi_1}=\ee^\Omega+\Theta+R,\quad \Theta\in\gg_{p+1},\ R\in\nR_{p+2}.
\end{equation}
For {\em self-adjoint} schemes $\ee^{\Phi_J}\cdots\ee^{\Phi_1}$ we may assume that $p$ is even, and the statement of the theorem holds already if the order conditions (\ref{eq:order_conditions}) are satisfied only for all Lyndon words of {\em odd} grade $q\leq p$.\footnote{Here we have to assume that $\ee^{\Omega}$ representing the exact solution operator is itself self-adjoint. For $\Omega=\AA+\BB$ or $\Omega=$ (truncated) Magnus series 
(\ref{eq:magnus_series}) this is of course the case.}
\end{theorem}

\begin{proof}
We apply a bootstrap argument:
Assume the statement of Theorem~\ref{thm:I} holds for $q=1$,
$$\ee^{\Phi_J}\cdots\ee^{\Phi_1}=\ee^\Omega+\Theta_1+R_2,\quad 0\neq\Theta_1\in\gg_1,\ R_2\in\nR_2.$$
Then the order conditions (\ref{eq:order_conditions}) related with $w\in\nW_1$ imply 
$\mathbf{c}_{\nW_1}=0$, 
thus $\mathbf{c}_{\nB_1}=T_1^{-1}\mathbf{c}_{\nW_1}=0$ 
(see (\ref{eq:coeff_vectors}))
 and hence
$\Theta_1=0$, contradicting the requirement $\Theta_1\neq 0$. 

It follows $q\geq 2$, so assume the statement of Theorem~\ref{thm:I} holds for $q=2$,
$$\ee^{\Phi_J}\cdots\ee^{\Phi_1}=\ee^\Omega+\Theta_2+R_3,\quad 0\neq\Theta_2\in\gg_2,\ R_3\in\nR_3.$$
Just as before,
the order conditions (\ref{eq:order_conditions}) related with $w\in\nW_2$ 
(or, in the self-adjoint case, simply the fact that the leading error term has odd order, see Theorem~\ref{thm:I_selfadjoint})
imply $\Theta_2=0$ if $p\geq 2$, again a contradiction. 

This reasoning can be iterated until $q=p$ and (\ref{eq:order_stmnt}) follows.
\qed
\end{proof}

A possible measure for the accuracy of a scheme of order $p$ as an approximation of the exact
solution operator $\ee^\Omega$
is the local error measure
\begin{equation}\label{eq:lem}
\mathrm{LEM}=\Vert\mathbf{c}_{\nW_{p+1}}\Vert = \sqrt{\sum_{w\in\nW_{p+1}}\!\!|c_w|^2},
\end{equation}
which is
built up from the coefficients $c_w=\coeff(w,\Theta)$ of all Lyndon words $w$ of grade $p+1$ in the leading
error term $\Theta$ from (\ref{eq:leading_error_term}), cf.\ \cite[Section~4]{part1}.

\subsection{Algorithm for computing coefficients of words}
\label{subsct:homos}
We will now derive an effective algorithm for the computation of
$ \coeff(w,\ee^{\Phi_J}\cdots\ee^{\Phi_1})$ in (\ref{eq:coeff_w_theta}), which
is based on a suitably constructed family of homomorphism $\{\varphi_w:w\in\nA^*\}$.
For each word $w=w_1\cdots w_{\ell(w)}\in\nA^*$ of length $\ell(w)\geq 1$ over the alphabet $\nA=\{\AA,\BB\}$ or $\nA=\{\AA_1,\dots,\AA_K\}$
we define the map 
\begin{equation}\label{eq:phi_w}
\varphi_w: \mathbb{C}\langle\langle{\nA}\rangle\rangle\to\mathbb{C}^{(\ell(w)+1)\times(\ell(w)+1)}\ \mbox{ by }\
\varphi_w(X)_{i,j} = 
%\varphi_{i,j}^{(w)}), \
%\varphi_{i,j}^{(w)}=
\left\{\begin{array}{ll}
\coeff(w_{i:j-1},X),&\mbox{if $i<j$,}\\
\coeff(\Id,X),&\mbox{if $i=j$,}\\
0,&\mbox{if $i>j$.}
\end{array}\right.
\end{equation}
Here $w_{i:j-1}=w_iw_{i+1}\cdots w_{j-1}$ denotes the subword of $w$ of length $j-i$, starting at position $i$ and ending at position $j-1$.
\begin{theorem}\label{thm:III} The map $\varphi_w$ 
defined by (\ref{eq:phi_w})
is an algebra homomorphism $$\mathbb{C}\langle\langle{\nA}\rangle\rangle\to\mathbb{C}^{(\ell(w)+1)\times(\ell(w)+1)},$$ i.e.,
\begin{enumerate}[(i)]
\item $\varphi_w$ is linear,
\begin{equation*}
\varphi_w(\alpha X+\beta Y)=\alpha\varphi_w(X)+\beta\varphi_w(Y),\quad
X,Y\in \mathbb{C}\langle\langle{\nA}\rangle\rangle,\ \alpha,\beta\in\mathbb{C};
\end{equation*}
\item $\varphi_w$ preserves the multiplicative structure,
\begin{equation*}
\varphi_w(X\cdot Y)=\varphi_w(X)\cdot\varphi_w(Y),\quad X,Y\in\mathbb{C}\langle\langle{\nA}\rangle\rangle.
\end{equation*}
\end{enumerate}
Furthermore, 
if $X\in\gg=[\mathbb{C}\langle\nA\rangle]$ (or more generally, 
if $X\in\mathbb{C}\langle\langle{\nA}\rangle\rangle$, % with 
 $\coeff(\Id, X)=0$), then  it holds
\begin{equation}\label{eq:exp_phi}
\varphi_w(\exp X) = \exp\varphi_w(X),
\end{equation}
where the exponential of the strictly upper triangular and thus nilpotent matrix $\varphi_w(X)$ is exactly computable in
a finite number of steps.
\end{theorem}
\begin{proof}
(i) trivial.
Ad (ii):
for $i<j$ it holds
\begin{align*}
\big(\varphi_w(X)\cdot&\varphi_w(Y)\big)_{i,j}
= \sum_{k=1}^{\ell(w)+1}\varphi_w(X)_{i,k}\,\varphi_w(Y)_{k,j} \\
&= \varphi_w(X)_{i,i}\,\varphi_w(Y)_{i,j}+
    \sum_{k=i+1}^{j-1}\varphi_w(X)_{i,k}\,\varphi_w(Y)_{k,j}+
    \varphi_w(X)_{i,j}\,\varphi_w(Y)_{j,j}\\
&=\coeff(\Id,X)\,\coeff(w_{i:j-1},Y)+
\sum_{k=i+1}^{j-1} \coeff(w_{i:k-1},Y)\,\coeff(w_{k:j-1},Y)\\
&\qquad+\coeff(\Id,X)\,\coeff(w_{i:j-1},Y)\\
&=\coeff(w_{i:j-1},X\cdot Y),
\end{align*}
and for the other cases,
$$\big(\varphi_w(X)\cdot\varphi_w(Y)\big)_{i,i} = \coeff(\Id,X)\,\coeff(\Id,Y)=
\coeff(\Id,X\cdot Y)$$
and $\big(\varphi_w(X)\cdot\varphi_w(Y)\big)_{i,j}=0$ for $i>j$.

Finally we prove (\ref{eq:exp_phi}).
Let $q=\grade(w)$ and $R\in\nR_{q+1}$ such that 
$\ee^X = \Id+X+\ldots+\tfrac{1}{q!}X^q+R.$ Then
\begin{align*}
\varphi_w(\exp X)&= \varphi_w\big(
\Id+X+\ldots+\tfrac{1}{q!}X^q+R
\big)\\
&=I_{\ell(w)+1}+\varphi_w(X)+\ldots+\tfrac{1}{q!}\varphi_w(X)^q+\underbrace{\varphi_w(R)}_{=0}\\
&=I_{\ell(w)+1}+\varphi_w(X)+\ldots+\tfrac{1}{\ell(w)!}\varphi_w(X)^{\ell(w)}\\
&=\exp\varphi_w(X),
\end{align*}
where we have used that $\varphi_w(X)$ is 
a strictly upper triangular matrix and thus
nilpotent of order
$\ell(w)+1$,
$\varphi_w(X)^{\ell(w)+1}=0$,
 and $\ell(w)\leq q=\grade(w)$. \qed
\end{proof}
For generators $g\in\nA$, $\nA=\{\AA,\BB\}$ or $\nA=\{\AA_1,\dots,\AA_K\}$
 it holds
\begin{equation*}
\varphi_w(g) = \mathrm{superdiag}(\gamma_1^{(w)},\ldots,\gamma_{\ell(w)}^{(w)})
=\left(
\begin{array}{ccccc}
0 & \gamma_1^{(w)} \\
 & 0 & \gamma_2^{(w)}\\
 &   & \ddots & \ddots \\
 &   & & 0 & \gamma_{\ell(w)}^{(w)} \\
 &   & &   & 0
\end{array}\right)
%\in\mathbb{C}^{(\ell(w)+1)\times(\ell(w)+1)}
\end{equation*}
with 
\begin{equation}\label{ex:gamma}
\gamma_j^{(w)}=\left\{\begin{array}{ll}
1 & \mbox{if $w_j=g$},\\
0 & \mbox{otherwise},
\end{array}\right.
\qquad j=1,\dots,\ell(w),\quad w=w_1\cdots w_{\ell(w)}.
\end{equation}
Starting from the values of $\varphi_w(g)$ for generators $g\in\nA$, %, $\nA=\{\AA,\BB\}$ or $\nA=\{\AA_1,\dots,\AA_K\}$, 
Theorem~\ref{thm:III} shows that
$\varphi_w(\Phi_j)$, $\Phi_j\in\gg=[\mathbb{C}\langle\nA\rangle]$ and further
$$\varphi_w(\ee^{\Phi_J}\cdots\ee^{\Phi_1})=\exp\big(\varphi_w(\Phi_J)\big)\cdots\exp\big(\varphi_w(\Phi_1)\big)$$
are well defined and can be effectively computed.
Extracting the entry in the right upper corner,
\begin{equation*}
\coeff(w,\ee^{\Phi_J}\cdots\ee^{\Phi_1}) = \varphi_w(\ee^{\Phi_J}\cdots\ee^{\Phi_1})_{1,\ell(w)+1}, %= \varphi_w(\ee^{\Phi_J}\cdots\ee^{\Phi_1})\cdot(0,\dots,0,1)^T,
\end{equation*}
leads to an efficient algorithm for the computation of $\coeff(w,\ee^{\Phi_J}\cdots\ee^{\Phi_1})$ in (\ref{eq:coeff_w_theta}).
To be more specific, the coefficient of a word $w\in\nA^*$ in an expression $X\in \mathbb{C}\langle\langle{\nA\rangle\rangle}$ can be calculated as
\begin{equation}\label{ex:first_component}
\coeff(w, X) = \text{first component of } \varphi(w, X,(0,\dots,0,1)^T),
\end{equation}
where the function 
\begin{equation*}
\varphi(w, X, y) = \varphi_w(X)\cdot y\in\mathbb{C}^{\ell(w)+1},\ X\in\mathbb{C}\langle\langle\nA\rangle\rangle,\ y=(y_0,y_1,\dots,y_{\ell(w)})^T\in\mathbb{C}^{\ell(w)+1}
\end{equation*}
can be evaluated recursively, 
\begin{equation}\label{eq:homv}
\varphi(w, X, y)=\left\{\begin{array}{ll}
\big(y_1\gamma^{(w)}_1,\dots,y_{\ell(w)}\gamma^{(w)}_{\ell(w)}, 0\big)^T, &                                  \mbox{if }X=g\in\nA,\ \gamma_j^{w} \mbox{ as in } (\ref{ex:gamma}), \\
\varphi(w,Y, y)+\varphi(w,Z, y), & \mbox{if }X=Y+Z, \\
\varphi(w,Y, \varphi(w, Z, y)), & \mbox{if }X=Y\cdot Z, \\
\varphi(w,Y\cdot Z-Z\cdot Y,y), & \mbox{if } X=[Y,Z], \\ 
\alpha\varphi(w, Y,y), & \mbox{if } X=\alpha Y,\ \alpha\in\mathbb{C},\footnotemark \\
\text{result of algorithm (\ref{eq:expY}) below},& \mbox{if }X=\exp(Y),\ \coeff(\Id,Y)=0.
\end{array}\right.
\end{equation}
\footnotetext{In an implementation of (\ref{eq:homv}) in a computer algebra system, the  case  $X=\alpha Y$ usually comprises not only  numbers $\alpha$, but also  symbols.}
Here, for $X=\exp(Y)$ the  following algorithm is applied:
\begin{equation}\label{eq:expY}
\begin{minipage}{0.72\textwidth}
\begin{tabbing}
 \qquad\= \kill
 {\bf input:} $w,Y,y$ \\
 {\bf output:} $z=\varphi(w,\exp(Y),y)$\\
  $h = y$;
  $z = y$;
  $\lambda = 1$\\
  {\bf for} $j=1:\ell(w)$ \\
\>  $\lambda = \lambda/j$\\
\>  $h = \varphi(w,Y, h)$\\
\>  $z = z +\lambda\cdot h$\\
  {\bf end }
\end{tabbing}
\end{minipage}
\end{equation}
It is clear that the algorithm (\ref{ex:first_component})--(\ref{eq:expY}) 
 can very easily be implemented in a computer algebra system.
 
\medskip

\noindent{\em Remark.} Considering the other components 
%of $\varphi(w, \ee^{\Phi_J}\cdots\ee^{\Phi_1},(0,\dots,0,1)^T)$ 
in (\ref{ex:first_component}) we obtain more generally the coefficients
$$\coeff(w_{j:\ell(w)}, \ee^{\Phi_J}\cdots\ee^{\Phi_1})=
\varphi(w, \ee^{\Phi_J}\cdots\ee^{\Phi_1},(0,\dots,0,1)^T)_j
$$
of all right factors $w_{j:\ell(w)}$ of $w$ in 
$\ee^{\Phi_J}\cdots\ee^{\Phi_1}$.
This can be exploited if {\em all}  Lyndon words $w\in \bigcup_{q=1}^p\nW_q$ of grade $q\leq p$
have to be considered as for example in Theorem~\ref{thm:order_conditions}. 
If $w\in\nW_q^{\{\AA,\BB\}}\setminus\{\AA\}$  
or $w\in\nW_q^{\{\AA_1,\dots,\AA_q\}}\setminus\{\AA_1\}$ 
is a Lyndon word of grade $q<p$, 
then respectively  $\underbrace{\AA\cdots\AA}_{p-q} w$ or $\underbrace{\AA_1\cdots\AA_1}_{p-q}w$ is a Lyndon word of grade $p$, and thus
$w$ is a proper right factor of a Lyndon word of grade $p$.
It follows that further efficiency can be gained by computing $\varphi(w,\ee^{\Phi_J}\cdots\ee^{\Phi_1},(0,\dots,0,1)^T)$ only 
for the Lyndon words $w\in\nW_p$ of grade exactly $p$, from which the coefficients 
$\coeff(w, \ee^{\Phi_J}\cdots\ee^{\Phi_1})$ for all $w\in \bigcup_{q=1}^p\nW_q$ 
($w\neq\AA$ and $w\neq\AA_1$) can be read off. The (very simple) exceptional cases $w=\AA$, $w=\AA_1$ 
have to be considered separately.

\section{Applications to (generalized) splitting methods}\label{sct:splitting}
\subsection{Exact solution}
We compute $\coeff(w, \ee^{\AA+\BB})$, where
$\ee^{\Omega}=\ee^{\AA+\BB}$ represents the exact solution operator for 
(\ref{eq:split_evolution_eq}).
This can be achieved by a very simple application of the homomorphisms $\varphi_w$.
Using $\varphi_w(\AA+\BB)=\varphi_w(\AA)+\varphi_w(\BB)=\mathrm{superdiag}(1,\dots,1)$
we obtain
\begin{align*}
\varphi_w(\ee^{\AA+\BB})
&=\exp\varphi_w(\AA+\BB)
=\exp\mathrm{superdiag}(1,\ldots,1)
%=\exp\left(
%\begin{array}{ccccc}
%0 & 1 \\
% & 0 & 1\\
% &   & \ddots & \ddots \\
% &   & & 0 & 1
%\end{array}\right)
=\left(\begin{array}{cccccc}
1 & 1 & \tfrac{1}{2} &  & \cdots & \tfrac{1}{\ell(w)!}\\
 & 1 & 1 &\tfrac{1}{2} \\
 &   & \ddots & \ddots & \ddots & \vdots\\
 &   & & 1 & 1 &\tfrac{1}{2}\\
 &   & &   & 1 & 1\\
 &   & &   &   & 1
\end{array}
\right),
\end{align*}
from which we read off
\begin{equation*}
\coeff(w, \ee^{\AA+\BB}) %= \coeff(w, \ee^{\AA+\BB}) 
%=\big(\varphi_w(\ee^{\AA+\BB})\big)_{1,\ell(w)+1}
=\frac{1}{\ell(w)!}, \quad w=w_1\cdots w_{\ell(w)}\in\nA^* \mbox{ for } \nA=\{\AA,\BB\}.
\end{equation*}

\subsection{Generalized fourth order splitting method}
In this subsection we derive the generalized fourth order splitting method
(\ref{eq:gs4}) and compute its  leading error term.
We make the ansatz
\begin{equation*}
\nS=\ee^{b\BB}\,\ee^{a\AA}\,\ee^{c\BB+d[\BB,[\AA,\BB]]}\,\ee^{a\AA}\,\ee^{b\BB},
\end{equation*}
which is symmetric in the sense of Section~\ref{subsct:symmetry}. The order
conditions for order $p=4$ are
\begin{equation*}
\coeff(w,\nS-\ee^{\AA+\BB})=0,\quad w\in\{\AA,\BB,\AA\AA\BB,\AA\BB\BB\},
\end{equation*}
where due to symmetry only Lyndon words of odd length $\leq 4$ have to be
considered, cf.\ Theorem~\ref{thm:order_conditions}. An application of   
algorithm (\ref{ex:first_component})--(\ref{eq:expY})
yields 4 equations
\begin{align*}
&2a-1=0, \quad
 2b+c-1=0, \\                   
& 2a^2b+\tfrac{1}{2}a^2c-\tfrac{1}{6}=0, \quad
 ab^2+\tfrac{1}{2}ac^2+abc-d-\tfrac{1}{6}=0
\end{align*} 
in 4 variables $a,b,c,d$, which have the unique solution
\begin{equation*}
a=\tfrac{1}{2}, \quad
b=\tfrac{1}{6}, \quad
c=\tfrac{2}{3}, \quad
d=\tfrac{1}{72},
\end{equation*}
i.e.,
\begin{equation*}
\nS=\ee^{\frac{1}{6}\BB}\ee^{\frac{1}{2}\AA}\ee^{\frac{2}{3}\BB+\frac{1}{72}[\BB,[\AA,\BB]]}\ee^{\frac{1}{2}\AA}\ee^{\frac{1}{6}\BB},
\end{equation*}
which confirms (\ref{eq:gs4}).

To determine the leading local error term 
%$\Theta$ (cf.\ Theorem~\ref{thm:I_selfadjoint}) 
we compute
the coefficients of the Lyndon words of length $p+1=5$,
$$
\mathbf{c}_{\nW_5}=\left(\begin{array}{c}
c_{\AA\AA\AA\AA\BB}\\
c_{\AA\AA\AA\BB\BB}\\
c_{\AA\AA\BB\AA\BB}\\
c_{\AA\AA\BB\BB\BB}\\
c_{\AA\BB\AA\BB\BB}\\
c_{\AA\BB\BB\BB\BB}
\end{array}\right)=
\left(\begin{array}{c}
\frac{1}{2880} \\[2pt]
  \frac{-7}{8640} \\[2pt]  
   \frac{1}{480}   \\[2pt]
   \frac{7}{12960} \\[2pt]
  \frac{-1}{720}   \\[2pt]
 \frac{-41}{155520}
\end{array}\right),
\quad
\mbox{where }c_w=\coeff(w,\nS-\ee^{\AA+\BB}),\footnote{For the local error measure (\ref{eq:lem}) we obtain
$\mathrm{LEM}=\|\mathbf{c}_{\nW_5} \|_2 \doteq 0.002721$ which is an order of magnitude smaller than the  
local error measures of the best 4th order classical (i.e., commutator-free) splitting methods involving, e.g., 
10 exponentials, cf.\ \cite[Sections~4, 5]{part1}.
}
$$
from which, via (\ref{eq:coeff_vectors}), we obtain
$$
\mathbf{c}_{\nB_5}=\left(\begin{array}{c}
c_{[\AA,[\AA,[\AA,[\AA,\BB]]]]}\\
c_{[\AA,[\AA,[[\AA,\BB],\BB]]]}\\
c_{[[\AA,[\AA,\BB]],[\AA,\BB]]}\\
c_{[\AA,[[[\AA,\BB],\BB],\BB]]}\\
c_{[[\AA,\BB],[[\AA,\BB],\BB]]}\\
c_{[[[[\AA,\BB],\BB],\BB],\BB]}
\end{array}\right)= T_5^{-1}\cdot \mathbf{c}_{\nW_5}=
\left(\begin{array}{c}
\frac{1}{2880} \\[2pt]
  \frac{-7}{8640} \\[2pt]  
   \frac{1}{2160}   \\[2pt]
   \frac{7}{12960} \\[2pt]
  \frac{1}{4320}   \\[2pt]
 \frac{-41}{155520}
\end{array}\right),
$$
and thus (substituting $\AA\to\tau A$, $\BB\to\tau B$) the representation of
the local error of (\ref{eq:gs4}),
\begin{align*}
\nL(\tau)&
=\ee^{\frac{1}{6}\tau B}\,\ee^{\frac{1}{2}\tau A}\,\ee^{\frac{2}{3}\tau B+\frac{1}{72}\tau^3[B,[A,B]]}
	\,\ee^{\frac{1}{2}\tau A}\,\ee^{\frac{1}{6}\tau B}
	\ - \ \ee^{\tau(A+B)}\\
	&=\tau^5\Big(
	        \tfrac{1}{2880}[A,[A,[A,[A,B]]]]-\tfrac{7}{8640}[A,[A,[[A,B],B]]]\\
	&\qquad+\tfrac{1}{2160}[[A,[A,B]],[A,B]]+\tfrac{7}{12960}[A,[[[A,B],B],B]]\\
	&\qquad+\tfrac{1}{4320}[[A,B],[[A,B],B]]-\tfrac{41}{155520}[[[[A,B],B],B],B]\Big) + \Order(\tau^6).
\end{align*}

For theoretical (and aesthetic) reasons it might be favorable to represent the
leading error term in the right normed basis 
\begin{align*}
\widetilde{\nB}_5 = \big\{\ &
[\AA,[\AA,[\AA,[\AA,\BB]]]],\
[\BB,[\AA,[\AA,[\AA,\BB]]]],\
[\AA,[\AA,[\BB,[\AA,\BB]]]],\\
&[\BB,[\AA,[\BB,[\AA,\BB]]]],\
[\AA,[\BB,[\BB,[\AA,\BB]]]],\
[\BB,[\BB,[\BB,[\AA,\BB]]]]
\ \big\}
\end{align*}
instead of the Lyndon basis $\nB_5$.
A procedure of constructing such bases is given in  \cite{CHIBRIKOV2006593}.
Instead of (\ref{eq:T_splitting}) the transformation matrix is now
$$\widetilde{T}_5=T_{\widetilde{\nB}_{5}\to\nW_{5}} = \big(\coeff(w,b)\big)_{w\in\nW_5,b\in\widetilde{\nB}_5}=
\left(
%\begin{array}{rrrrrr}
% 1 &  0 &  0 &  0 &  0 &  0\\
% 0 & -1 & -1 &  0 &  0 &  0\\
% 0 &  3 &  2 &  0 &  0 &  0\\
% 0 &  0 &  0 &  1 &  1 &  0\\
% 0 &  0 &  0 & -2 & -3 &  0\\
% 0 &  0 &  0 &  0 &  0 & -1
%\end{array}
\begin{array}{rrrrrr}
 1 &    &    &    &    &   \\
   & -1 & -1 &    &    &   \\
   &  3 &  2 &    &    &   \\
   &    &    &  1 &  1 &   \\
   &    &    & -2 & -3 &   \\
   &    &    &    &    & -1
\end{array}
\right),
$$
where the entries can  be computed using 
algorithm (\ref{ex:first_component})--(\ref{eq:expY}).
We obtain 
$$
\mathbf{c}_{\widetilde{\nB}_5}=\left(\begin{array}{c}
c_{[\AA,[\AA,[\AA,[\AA,\BB]]]]}\\
c_{[\BB,[\AA,[\AA,[\AA,\BB]]]]}\\
c_{[\AA,[\AA,[\BB,[\AA,\BB]]]]}\\
c_{[\BB,[\AA,[\BB,[\AA,\BB]]]]}\\
c_{[\AA,[\BB,[\BB,[\AA,\BB]]]]}\\
c_{[\BB,[\BB,[\BB,[\AA,\BB]]]]}
\end{array}\right)= \widetilde{T}_5^{-1}\cdot \mathbf{c}_{\nW_5}=
\left(\begin{array}{c}
  \frac{1}{2880} \\[2pt]
  \frac{1}{2160}  \\[2pt]
  \frac{1}{2880}  \\[2pt]
  \frac{1}{4320}  \\[2pt]
  \frac{1}{3240}  \\[2pt]
 \frac{41}{155520}
\end{array}
\right)
$$
and thus
\begin{align*}
\nL(\tau)&=
	\tau^5\Big(\tfrac{1}{2880}[A,[A,[A,[A,B]]]]+\tfrac{1}{2160}[B,[A,[A,[A,B]]]]\\
	&\qquad+\tfrac{1}{2880}[A,[A,[B,[A,B]]]]+\tfrac{1}{4320}[B,[A,[B,[A,B]]]] \\
	&\qquad+\tfrac{1}{3240}[A,[B,[B,[A,B]]]]+\tfrac{41}{155520}[B,[B,[B,[A,B]]]]\Big)\ + \Order(\tau^6).
\end{align*}	

%\begin{align*}
%\Theta =&\ \ \ \
%\tfrac{1}{2880}[\AA,[\AA,[\AA,[\AA,\BB]]]]+
%\tfrac{1}{2160} [\BB,[\AA,[\AA,[\AA,\BB]]]]+
%\tfrac{1}{2880}[\AA,[\AA,[\BB,[\AA,\BB]]]] \\
%&+\tfrac{1}{4320} [\BB,[\AA,[\BB,[\AA,\BB]]]]+
%\tfrac{1}{3240}[\AA,[\BB,[\BB,[\AA,\BB]]]]+
%\tfrac{41}{155520}[\BB,[\BB,[\BB,[\AA,\BB]]]],
%\end{align*}
%which verifies the representation (\ref{eq:gen_split_4}).

\section{Applications to Magnus-type methods}\label{sct:magnus}
\subsection{Exact solution}
For non-autonomous problems (\ref{eq:non_auto_evolution_eq}) it is not as straightforward as for the
splitting case
to provide an explicit formula
for $\mathrm{coeff}(w,\ee^{\Omega})$, where $\ee^{\Omega}$ represents the exact solution operator for (\ref{eq:non_auto_evolution_eq}).
Of course, using algorithm (\ref{ex:first_component})--(\ref{eq:expY}), we could compute  $\mathrm{coeff}(w,\ee^{\Omega})$  
from the Magnus series (\ref{eq:magnus_series}). However, we prefer to derive  an explicit formula for
$\mathrm{coeff}(w,\ee^{\Omega})$ which is not based on an explicit representation of the 
Magnus series $\Omega$. 

\begin{theorem}
Let $w=\AA_{d_1}\cdots\AA_{d_\ell}$ be a word over the alphabet $\{\AA_1,\AA_2,\dots\}$,
where the $\AA_k\simeq A_k$ represent the Legendre coefficients $A_k$ from (\ref{eq:A_expansion}). Then the coefficient of $w$ in $\ee^\Omega$ representing the
exact solution operator for (\ref{eq:non_auto_evolution_eq}) is given by 
\begin{equation}\label{eq:coeff_rhs_magnus_type}
\mathrm{coeff}(\AA_{d_1}\cdots\AA_{d_\ell},\ee^{\Omega}) %=\mathrm{coeff}(w_1\dots w_{n_w},\ee^{\Omega})
=\sum_{(k_1,\dots,k_{\ell})\atop 1\leq k_l\leq d_l}
\prod_{j=1}^{\ell}\frac{
(-1)^{d_j+k_j}{d_j-1 \choose k_j-1}{d_j+k_j-2 \choose k_j-1}
}{\sum_{i=j}^{\ell}k_i}.
\end{equation}
\end{theorem}
\begin{proof}
Define
\begin{equation}\label{eq:hatAk}
\hat{A}_k=\sum_{d\geq k}(-1)^{d+k}{d-1 \choose k-1}{d+k-2 \choose k-1}A_d
, \quad k=1,2,\dots.
\end{equation}
Then from (\ref{eq:A_expansion}) and (\ref{eq:legendre_zeugs}) it follows
\begin{equation}\label{eq:A_hat_expansion}
A(t) = \sum_{d\geq 1} \widetilde{P}_{d-1}(t)A_d = \frac{1}{\tau}\sum_{k\geq 1}\left(\frac{t}{\tau}\right)^{k-1}\hat{A}_k,
\end{equation}
which is a Taylor expansion of $A(t)$ and thus
\begin{equation}\label{eq:taylor_coeff}
\hat{A}_k=\hat{A}_k(\tau)=\frac{\tau^k}{(k-1)!}\left.\frac{d^{k-1}}{dt^{k-1}}A(t)\right|_{t=0}, \quad k=1,2,\dots.
\end{equation}
It is easy to see that the relations (\ref{eq:hatAk}) can be inverted such that the $A_d$ can be written as (finite) linear combinations of the $\hat{A}_k$. Substituting these representations of the
$A_d$ in 
\begin{equation}\label{eq:exact_solution_expandion}
u(\tau)=\ee^{\Omega}u(0)= \sum_{w=A_{d_1}\cdots A_{d_\ell}}c_{w}\,A_{d_1}\cdots A_{d_\ell}\,u(0),\quad c_{w}=\coeff( A_{d_1}\cdots A_{d_\ell},\,\ee^\Omega)
\end{equation}
and expanding we obtain
$$
u(\tau)= \sum_{\hat{w}=\hat{A}_{k_1}\cdots \hat{A}_{k_\ell}}c_{\hat{w}}\,\hat{A}_{k_1}\cdots \hat{A}_{k_\ell}\,u(0)%,\quad c_{\hat{w}}=\coeff( \hat{A}_{k_1}\cdots \hat{A}_{k_\ell},\,\ee^\Omega)
$$
with well-defined coefficients $c_{\hat{w}}$.
Using $\hat{A}_k(t)=(\frac{t}{\tau})^k\hat{A}_k(\tau)$ (cf.~(\ref{eq:taylor_coeff})) we obtain an
expansion of the exact solution
of (\ref{eq:non_auto_evolution_eq}),
$$
u(t) = \sum_{\hat{w}=\hat{A}_{k_1}\cdots \hat{A}_{k_\ell}}\left(\frac{t}{\tau}\right)^{\sum_{j=1}^{\ell} k_j}\!c_{\hat{w}}\,\hat{A}_{k_1}\cdots \hat{A}_{k_\ell}\,u(0),\quad t\geq 0
\quad (\hat{A}_{k_j}=\hat{A}_{k_j}(\tau)),
$$
such that on one hand
$$
u'(t) = \frac{1}{\tau}\sum_{\hat{w}=\hat{A}_{k_1}\cdots \hat{A}_{k_\ell}}\left(\frac{t}{\tau}\right)^{\sum_{j=1}^\ell k_j-1}\left(\sum_{j=1}^\ell k_j\right)c_{\hat{w}}\,\hat{A}_{k_1}\cdots \hat{A}_{k_\ell}\,u(0),
$$
and on the other hand, using (\ref{eq:A_hat_expansion}),
\begin{align*}
u'(t)=A(t)u(t)&=\frac{1}{\tau}\left(\sum_{k_1\geq 1}\left(\frac{t}{\tau}\right)^{k_1-1}\!\hat{A}_{k_1}\right)\cdot\left(\sum_{\hat{v}=\hat{A}_{k_2}\cdots \hat{A}_{k_\ell}}\!\!\!\left(\frac{t}{\tau}\right)^{\sum_{j=2}^{\ell} k_j}c_{\hat{v}}\,\hat{A}_{k_2}\cdots \hat{A}_{k_\ell}\right)u(0)\\
&=\frac{1}{\tau} \sum_{\hat{w}=\hat{A}_{k_1}\cdots \hat{A}_{k_\ell}}
\!\!\!\!\left(\frac{t}{\tau}\right)^{\sum_{j=1}^\ell k_j-1}\!c_{\hat{w}_{2:\ell}}\,\hat{A}_{k_1}\cdots \hat{A}_{k_\ell}\, u(0).
\end{align*}
Comparing corresponding coefficients of $\hat{w}= \hat{A}_{k_1}\cdots \hat{A}_{k_\ell}$ in the two expressions for $u'(t)$ we
obtain recursively
$$
c_{\hat{w}}=\frac{1}{\sum_{i=1}^{\ell}k_i}c_{\hat{w}_{2:\ell}}=
\frac{1}{\sum_{i=1}^{\ell}k_i}\cdot\frac{1}{\sum_{i=2}^{\ell}k_i}c_{\hat{w}_{3:\ell}}=
\ldots =
\prod_{j=1}^{\ell}\frac{1}{\sum_{i=j}^{\ell}k_i}.
$$
Substituting (\ref{eq:hatAk}) in 
$$ u(\tau)=
\sum_{\hat{w}=\hat{A}_{k_1}\cdots \hat{A}_{k_\ell}}c_{\hat{w}}\hat{A}_{k_1}\cdots \hat{A}_{k_\ell}\,u(0)=
\sum_{k_1\geq 1,\dots, k_\ell\geq 1}\prod_{j=1}^{\ell}
\frac{\hat{A}_j}{\sum_{i=j}^{\ell}k_i}u(0)
$$
and comparing coefficients with (\ref{eq:exact_solution_expandion})
we obtain (\ref{eq:coeff_rhs_magnus_type}). \qed
\end{proof}

\begin{table}[t!]
\begin{center}
\begin{tabular}{|c|c|c|cc|ccc|}%|cccccc|}
\hline
$q$ & 1 & 2 & 3 & &  4 & & %& %5 & & & & & 
\\
\hline
$w$                    & $\AA_1$ & $\AA_2$ & $\AA_1\AA_2$ &  $\AA_3$ & $\AA_1\AA_1\AA_2$ & $\AA_1\AA_3$ & $\AA_4$  %&$\AA_1\AA_1\AA_1\AA_2$ & $\AA_1\AA_1\AA_3$ & $\AA_1\AA_2\AA_2$ & $\AA_1\AA_4$ & $\AA_2\AA_3$ & $\AA_5$
\\
$\coeff(w,\ee^\Omega)$& $1$     & $0$     & $-\frac{1}{6}$ &  $0$     & $-\frac{1}{12}$ &  $0$         & $0$   % &$-\frac{1}{40}$ &
% $\frac{1}{60}$ &
%  $\frac{1}{60}$ & 
%  $0$ &
% $-\frac{1}{30}$ & 
%5  $0$ 
\\
\hline  
\end{tabular}

\vspace{1mm}

\begin{tabular}{|c|cccccc|}
\hline
$q$ & 5 & & & & & \\  
\hline
$w$ &$\AA_1\AA_1\AA_1\AA_2$ & $\AA_1\AA_1\AA_3$ & $\AA_1\AA_2\AA_2$ & $\AA_1\AA_4$ & $\AA_2\AA_3$ & $\AA_5$ \\
$\coeff(w,\ee^\Omega)$&
$-\frac{1}{40}$ &
 $\frac{1}{60}$ &
  $\frac{1}{60}$ & 
  $0$ &
 $-\frac{1}{30}$ & 
  $0$ \\
\hline
\end{tabular}
\caption{Coefficients $\coeff(w,\ee^\Omega)$ of Lyndon words $w$ of grade $q \leq 5$ in the exact solution operator $\ee^\Omega$.
\label{tab:rhs_magnus_type}
}
\end{center}
\end{table}
Some values of (\ref{eq:coeff_rhs_magnus_type}) are shown in Table~\ref{tab:rhs_magnus_type}. 
Using (\ref{eq:legendre_zeugs}) it is easy to  see that it holds
\begin{equation}\label{eq:l_fold_integral}
\mathrm{coeff}(\AA_{d_1}\cdots\AA_{d_\ell},\ee^{\Omega}) 
=\int_0^1\int_0^{x_1}\cdots \int_0^{x_{\ell-1}}P_{d_1-1}(x_1)\cdots P_{d_\ell-1}(x_\ell)\, \dd x_\ell\cdots \dd x_2\, \dd x_1.
\end{equation}
From the orthogonality of the Legendre polynomials it follows that 
$\mathrm{coeff}(\AA_{d_1}\cdots\AA_{d_\ell},\ee^{\Omega})$ vanishes if some index $d_j$ exceeds
the sum of the others by at least two, see \cite[Section~3.2]{alvfeh11} where the $\ell$-fold integral (\ref{eq:l_fold_integral}) is denoted $\xi(d_1,\dots, d_\ell)$.
In particular,
\begin{equation*}
\coeff(w,\ee^\Omega)=0,\quad \mbox{if $\grade(w)\leq p$ and $w$ contains $\AA_d$ with $d\geq\tfrac{p}{2}+1$},
\end{equation*}
which implies that
if a scheme $\nS=\ee^{\Phi_J}\cdots\ee^{\Phi_1}$ does not involve any generator $A_d$ with 
$d\geq\tfrac{p}{2}+1$, then the order conditions (\ref{eq:order_conditions}) of Theorem~\ref{thm:order_conditions} are automatically satisfied for all Lyndon words $w$ of grade $\leq p$ which contain such an 
$\AA_d$. Thus, the number of order conditions to be considered for such schemes is 
significantly reduced.
However, for $p$ even, the coefficients 
\begin{equation*}
c_w=%\coeff(w,\Theta) %=
 \coeff(w,\nS-\ee^\Omega)=-\coeff(w,\ee^\Omega),\quad w\in\nW_{p+1},\ w\text{ contains }\AA_{\frac{p}{2}+1}
\end{equation*}
in the leading local error term
do not necessarily vanish, which implies
a lower bound 
\begin{equation}\label{eq:lem_bound}
\mathrm{LEM}\geq\sqrt{\sum_{w\in\nW_{p+1},\ w\text{ contains }\AA_{\frac{p}{2}+1}}|c_w|^2}
\qquad\text{($p$ even)}
\end{equation}
for the local error measure (\ref{eq:lem}) of
a scheme not involving $\AA_{\frac{p}{2}+1}$. Therefore, an optimized scheme of even 
order $p$ with a LEM below this bound must necessarily involve the generator $\AA_{\frac{p}{2}+1}$, see also \cite[Section~5]{alvfeh11}.

\subsection{Fourth order commutator-free Magnus-type integrator}\label{subsct:cf4}
In this subsection we derive the fourth order Magnus-type integrator
(\ref{eq:comfree4}) and compute its  leading error term.
We make the ansatz
\begin{equation*}
\nS=\ee^{f_1\AA_1-f_2\AA_2}\,\ee^{f_1\AA_1+f_2\AA_2},
\end{equation*}
which is symmetric in the sense of Section~\ref{subsct:symmetry} and does not involve 
$\AA_3$ or $\AA_4$.
%This ansatz
%without $\AA_3$ or $\AA_4$ is reasonable because
%it holds $\coeff(w, \ee^{\Omega})=0$ for $w\in\{\AA_3,\, \AA_1\AA_3,\,\AA_4\}$, the Lyndon words  of  grade $\leq 4$ containing one of the generators $\AA_3$, $\AA_4$
%Due to symmetry only Lyndon words of odd grade have to be considered, 
%cf.\ Theorem~\ref{thm:order_conditions}.
%The order
%conditions for order $p=4$ corresponding to the Lyndon words of odd grade not containing 
%$\AA_3$ or $\AA_4$ are
%\begin{equation*}
%\coeff(w,\nS-\ee^{\Omega})=0,\quad w\in\{\AA_1,\, \AA_1\AA_2\}.
%\end{equation*}
The order conditions for order $p=4$ are 
\begin{equation*}
\coeff(w,\nS-\ee^{\Omega})=0,\quad w\in\{\AA_1,\, \AA_1\AA_2\}, 
\end{equation*}
where only Lyndon words 
of odd grade $\leq 4$ (due to symmetry, cf.\ Theorem~\ref{thm:order_conditions}) not containing
$\AA_3$ or $\AA_4$ have to be considered.
An application of   
algorithm (\ref{ex:first_component})--(\ref{eq:expY})
yields the equations
\begin{equation*}
2f_1-1=0, \quad
 f_1f_2+\tfrac{1}{6}=0
\end{equation*} 
with  solution
$f_1 =\tfrac{1}{2}$,
$f_2=-\tfrac{1}{3}$,
i.e.,
\begin{equation}\label{eq:comfree4formal}
\nS=\ee^{\frac{1}{2}\AA_1+\frac{1}{3}f_2\AA_2}\,\ee^{\frac{1}{3}\AA_1-\frac{1}{3}\AA_2},
\end{equation}
which confirms (\ref{eq:comfree4}).

To determine the leading local error term we compute, using   
algorithm (\ref{ex:first_component})--(\ref{eq:expY})
and then (\ref{eq:coeff_vectors}),
\begin{equation*}
\mathbf{c}_{\nW_5}=
\left(\begin{array}{c}
c_{\AA_1\AA_1\AA_1\AA_2}\\
c_{\AA_1\AA_1\AA_3}\\
c_{\AA_1\AA_2\AA_2}\\
c_{\AA_1\AA_4}\\
c_{\AA_2\AA_3}\\
c_{\AA_5}
\end{array}\right)
=
\left(\begin{array}{c}
\frac{1}{1440}\\[2pt]
\frac{-1}{60}\\[2pt]
\frac{1}{540}\\[2pt]
0\\
\frac{1}{30}\\[2pt]
0
\end{array}\right)\ \mbox{and}\ \
\mathbf{c}_{\nB_5}=
\left(\begin{array}{c}
c_{[\AA_1,[\AA_1,[\AA_1,\AA_2]]]}\\
c_{[\AA_1,[\AA_1\AA_3]]}\\
c_{[[\AA_1,\AA_2],\AA_2]}\\
c_{[\AA_1,\AA4]}\\
c_{[\AA_2,\AA_3]}\\
c_{\AA_5}
\end{array}\right)=\underbrace{T_5^{-1}}_{=I_6}\cdot\mathbf{c}_{\nW_5}=\mathbf{c}_{\nW_5}.
\end{equation*}
For the local error $\nL(\tau,t_n)=\nS(\tau,t_n)-\ee^{\Omega}$ of (\ref{eq:comfree4}) we thus obtain
\begin{align*}
\nL(\tau, t_n)&=\ee^{\frac{1}{2}A_1+\frac{1}{3}A_2}\,\ee^{\frac{1}{2}A_1-\frac{1}{3}A_2} -
\ee^{\Omega}\\
&=\tfrac{1}{1440}[A_1,[A_1,[A_1,A_2]]]-\tfrac{1}{60}[A_1,[A_1,A_3]]\\
&\quad+\tfrac{1}{540}[[A_1,A_2],A_2]+\tfrac{1}{30}[A_2,A_3]\ + \ \Order(\tau^6).
\end{align*}	
For the local error measure (\ref{eq:lem}) we obtain $\mathrm{LEM}=\Vert\mathbf{c}_{\nW_5}\Vert\doteq 0.03732$, which
is only insignificantly larger than the lower bound
$\sqrt{c_{\AA_1\AA_1\AA_3}^2+c_{\AA_2\AA_3}^2}\doteq 0.03727$ of (\ref{eq:lem_bound}).

In numerical applications the $A_k$ form (\ref{eq:A_integral}) have to be approximated by an appropriate 
quadrature formula. 
Thus we substitute 
\begin{equation}\label{eq:Al_substitution}
\AA_l\to(2l-1)\tau\sum_{k=1}^{K}w_kP_{l-1}(x_k)A(t_n+\tau x_k)
\end{equation}
in (\ref{eq:comfree4formal}) with nodes and weights of Gaussian quadrature of order four,
$$
(x_k)=\left(\tfrac{1}{2}-\tfrac{\sqrt{3}}{6},\ \tfrac{1}{2}-\tfrac{\sqrt{3}}{6}\right),
\quad (w_k)=\left(\tfrac{1}{2},\ \tfrac{1}{2}\right),
$$
%are nodes and weights of Gaussian quadrature of order four, 
and obtain the
commutator-free
Magnus-type integrator in terms of the system matrix $A(t)$ evaluated at different times,
\begin{equation*}
\nS_{A}(\tau, t_n)=\ee^{\tau a_{21}A(t_n+\tau x_1)+\tau a_{22}A(t_n+\tau x_2)}\,
\ee^{\tau a_{11}A(t_n+\tau x_1)+\tau a_{12}A(t_n+\tau x_2)}
\end{equation*}
with coefficients
$$
(a_{jk})=\left(\begin{array}{cc}
\frac{1}{4}+\frac{\sqrt{3}}{6} &\ \frac{1}{4}-\frac{\sqrt{3}}{6} \\[2pt]
\frac{1}{4}-\frac{\sqrt{3}}{6} &\ \frac{1}{4}+\frac{\sqrt{3}}{6}
\end{array}\right),
$$
cf.\ \cite[Section~7]{alvfeh11}.

\subsection{6th order Magnus-type integrator involving commutator}
We work through analogous steps as in the previous Subsection~\ref{subsct:cf4}.
To derive the Magnus-type integrator (\ref{eq:unconv_scheme}), (\ref{eq:coeff_unconv})
from \cite[Section~4.3]{SergioFernandoMPaper2}
 we
make the symmetric ansatz
\begin{align*}
\nS&=% \ee^{\Phi_5}\ee^{\Phi_4}\ee^{\Phi_3}\ee^{\Phi_2}\ee^{\Phi_1}\\
\ee^{f_{11}\AA_1-f_{12}\AA_2+f_{13}\AA_3}\,
\ee^{f_{21}\AA_1-f_{22}\AA_2+f_{23}\AA_3}\,
\ee^{[g_1\AA_1+g_3A_3,\AA_2]}\\
&\quad\times \ee^{f_{21}\AA_1+f_{22}\AA_2+f_{23}\AA_3}\,
\ee^{f_{11}\AA_1+f_{12}\AA_2+f_{13}\AA_3}.
\end{align*}
%It involves the generators $\AA_1$, $\AA_2$, $\AA_3$ only, which is appropriate to  $\coeff(x,\ee^\Omega)=0$ for
%the Lyndon words $w\in\{\AA_4,\, \AA_1\AA_4,\, \AA_5,\, \AA_1\AA_1\AA_4,\, \AA_1\AA_5,\, \AA_2\AA_4,\, \AA_6\}$ of order $\leq 6$ containing $\AA_4$, $\AA_5$, or $\AA_6$.
Corresponding to the Lyndon words of odd grade $\leq 6$ not containing $\AA_4$, $\AA_5$, $\AA_6$
we obtain the order conditions
\begin{equation*}%\label{eq:llxx}
\coeff(w,\nS-\ee^\Omega)=0,\quad w\in\{\AA_1,\
 \AA_1\AA_2,\ \AA_3,\ \AA_1\AA_1\AA_1\AA_2, \ \AA_1\AA_1\AA_3, \ \AA_1\AA_2\AA_2,\ \AA_2\AA_3 \}
\end{equation*}
for order $p=6$.
An application of   
algorithm (\ref{ex:first_component})--(\ref{eq:expY})
yields 7 equations
\begin{align*}
&2 f_{11}+2 f_{21}-1=0,\\                                                                                                                                                                                                                                                                                    
& f_{11} f_{12}+2 f_{12} f_{21}+f_{21} f_{22}+g_{1}+\tfrac{1}{6}=0,\\                                                                                                                                                                                                                                                                
& 2 f_{13}+2 f_{23}=0,\\                                                                                                                                                                                                                                                                                      
& \tfrac{4}{3} f_{12} f_{21}^3+\tfrac{7}{12} f_{11}^3 f_{12}+\tfrac{7}{12} f_{21}^3 f_{22}+3 f_{11} f_{12} f_{21}^2+f_{11} f_{21}^2 f_{22}+\tfrac{7}{3} f_{11}^2 f_{12} f_{21}\\
&\qquad+\tfrac{1}{2}f_{11}^2 f_{21} f_{22}
+\tfrac{1}{2} f_{11}^2 g_{1}+\tfrac{1}{2} f_{21}^2 g_{1}+f_{11} f_{21} g_{1}+\tfrac{1}{40}=0,\\                                                                                                                                          
& 2 f_{13} f_{21}^2+\tfrac{4}{3} f_{11}^2 f_{13}+f_{11}^2 f_{23}+\tfrac{4}{3} f_{21}^2 f_{23}+3 f_{11} f_{13} f_{21}+2 f_{11} f_{21} f_{23}-\tfrac{1}{60}=0,\\                                                                                                                                                                                                               
& \tfrac{1}{3} f_{11} f_{12}^2+\tfrac{1}{3} f_{21} f_{22}^2+f_{12}^2 f_{21}+f_{12} f_{21} f_{22}+f_{12} g_{1}+f_{22} g_{1}-\tfrac{1}{60}=0,\\                                                                                                                                                                                                                             
&-f_{12} f_{13}-2 f_{12} f_{23}-f_{22} f_{23}-g_{3}+\tfrac{1}{30}=0                                                                                                                                                                                                                                                   
\end{align*}
in the 8 variables $f_{11},f_{12},f_{13},f_{21},f_{22},f_{23},g_{1},g_{3}$.
Following \cite[Section~4.3]{SergioFernandoMPaper2} we add the  condition
$$\coeff(\AA_1\AA_1\AA_1\AA_1\AA_1\AA_2, \nS-\ee^\Omega)=0$$ leading to an eighth equation
\begin{align*}
& \tfrac{4}{15} f_{12} f_{21}^5+\tfrac{31}{360} f_{11}^5 f_{12}+\tfrac{31}{360} f_{21}^5 f_{22}+f_{11} f_{12} f_{21}^4+\tfrac{1}{4} f_{11} f_{21}^4 f_{22}+\tfrac{14}{9} f_{11}^2 f_{12} f_{21}^3\\
&\qquad+\tfrac{7}{24} f_{11}^2 f_{21}^3 f_{22}
+\tfrac{5}{4} f_{11}^3 f_{12} f_{21}^2+\tfrac{1}{6} f_{11}^3 f_{21}^2 f_{22}+\tfrac{31}{60} f_{11}^4 f_{12} f_{21}+\tfrac{1}{24} f_{11}^4 f_{21} f_{22}\\
&\qquad+\tfrac{1}{24} f_{11}^4 g_{1}+\tfrac{1}{24} f_{21}^4 g_{1}
+\tfrac{1}{6} f_{11} f_{21}^3 g_{1}+\tfrac{1}{4} f_{11}^2 f_{21}^2 g_{1}+\tfrac{1}{6} f_{11}^3 f_{21} g_{1}+\tfrac{1}{1008}=0.
\end{align*}
For this system of 8 equations,
a computer algebra system readily finds 5 solutions involving only real numbers and
2 solutions involving also complex numbers. One of the real solutions is given in
(\ref{eq:coeff_unconv}). 

For the corresponding scheme we obtain 
$\mathrm{LEM}\doteq 0.0167$, which, within the given level of precision, is equal to the lower bound  (\ref{eq:lem_bound}).

Performing the substitution (\ref{eq:Al_substitution}) in 
(\ref{eq:unconv_scheme}), (\ref{eq:coeff_unconv})
now with Gaussian nodes and weights
of order six,
$$
(x_k)=\left(\tfrac{1}{2}-\tfrac{\sqrt{15}}{10},\ \tfrac{1}{2},\ \tfrac{1}{2}+\tfrac{\sqrt{15}}{10}\right),
\quad (w_k)=\left(\tfrac{5}{18},\ \tfrac{4}{9},\ \tfrac{5}{18} \right),
$$
we obtain
\begin{align*}
\nS_A(\tau, t_n)&=\ee^{\tau a_{13}A(t_n+\tau x_1)+\tau a_{12}A(t_n+\tau x_2)+\tau a_{11}A(t_n+\tau x_3)}\\
&\quad\times\ee^{\tau a_{23}A(t_n+\tau x_1)+\tau a_{22}A(t_n+\tau x_2)+\tau a_{21}A(t_n+\tau x_3)}\\
&\quad \times\ee^{\tau^2b_{1}[A(t_n+\tau x_2),\,A(t_n+\tau x_3)-A(t_n+\tau x_1)]
+\tau^2b_{2}[A(t_n+\tau x_3),\,A(t_n+\tau x_1)]}\\
%+\tau^2b_{3}[A(t_n+\tau x_1),\,A(t_n+\tau x_2)]}\\
&\quad\times\ee^{\tau a_{21}A(t_n+\tau x_1)+\tau a_{22}A(t_n+\tau x_2)+\tau a_{23}A(t_n+\tau x_3)}\\
&\quad\times\ee^{\tau a_{11}A(t_n+\tau x_1)+\tau a_{12}A(t_n+\tau x_2)+\tau a_{13}A(t_n+\tau x_3)}
\end{align*}
with coefficients
\begin{align*}
(a_{jk}) \doteq\ & \left({\fontsize{9}{9.9}\selectfont\begin{array}{rrr}
 0.210034604487283585&
-0.059278594478107764&
 0.015842684397126231\\
 0.108253098901669707&
 0.281500816700329986&
-0.056352610008301747
\end{array}}\right),\\
 b_1\doteq \ &  0.000355878988200746,\quad
b_2\doteq -0.000421029282637892.
\end{align*}
%A=[0.210034604487283585 -0.059278594478107764  0.015842684397126231
%   0.108253098901669707  0.281500816700329986 -0.056352610008301747
   
\subsection{Commutator-free integrator of order 8 involving only 8 exponentials}
\label{subsct:eight}
In \cite[Section~4.4]{alvfeh11} an 8th order commutator-free Magnus-type integrator
 involving 11 exponentials is proposed.
 Here we report on the derivation of a new such  8th order integrator with essentially the same accuracy but  involving only 8 exponentials.
 
We make the symmetric ansatz
\begin{align*}
        \nS&=\ee^{f_{11}\AA_1-f_{12}\AA_2+f_{13}\AA_3-f_{14}\AA_4}
    \,\ee^{f_{21}\AA_1-f_{22}\AA_2+f_{23}\AA_3-f_{24}\AA_4}\,\\
        &\quad\times\ee^{f_{31}\AA_1-f_{32}\AA_2+f_{33}\AA_3-f_{34}\AA_4}
        \,\ee^{f_{41}\AA_1-f_{42}\AA_2+f_{43}\AA_3-f_{44}\AA_4}\,\\
        &\quad\times\ee^{f_{51}\AA_1-f_{52}\AA_2+f_{53}\AA_3-f_{54}\AA_4}
        \,\ee^{f_{61}\AA_1+f_{63}} 
        \,\ee^{f_{51}\AA_1+f_{52}\AA_2+f_{53}\AA_3+f_{54}\AA_4}\,\\
        &\quad\times\ee^{f_{41}\AA_1+f_{42}\AA_2+f_{43}\AA_3+f_{44}\AA_4}
        \,\ee^{f_{31}\AA_1+f_{32}\AA_2+f_{33}\AA_3+f_{34}\AA_4}\\
        &\quad\times\ee^{f_{21}\AA_1+f_{22}\AA_2+f_{23}\AA_3+f_{24}\AA_4}
    \,\ee^{f_{11}\AA_1+f_{12}\AA_2+f_{13}\AA_3+f_{14}\AA_4}
\end{align*} 
with 8 exponentials
containing 22 coefficients $f_{11},f_{12},\dots$.
Corresponding to the 22
Lyndon words of odd grade $\leq 8$ over the alphabet $\nA=\{\AA_1,\AA_2, \AA_3, \AA_4\}$,
there are 22 order conditions $\coeff(w, \nS-\ee^\Omega)=0$.

Our Maple implementation of algorithm (\ref{ex:first_component})--(\ref{eq:expY}) generates
the corresponding 22 polynomial equation of maximum degree 6 in less than one second on a current standard desktop PC. The size of the output containing the equations is 289 kilobytes.
Due to the considerable complexity of the equations, a certain sophistication is required
for the efficient calculation of solutions of these equations. Therefore,
we use the symbolic manipulation system FORM \cite{KUIPERS20131453}\footnote{\url{http://www.nikhef.nl/~form}} to generate highly
optimized C code for the evaluation of the equations and their Jacobi matrix. The size of the
resulting C code is 120 kilobytes, its generation takes less than 5 seconds, its compilation to
machine code with high optimization level less than 3 seconds.
A nonlinear solver from the Julia package NLsolve.jl\footnote{\url{https://github.com/JuliaNLSolvers/NLsolve.jl}}
applied to the thus pre-processed  equations,
which is repeatedly restarted with random starting values, 
% A simple implementation of
%Newton's method, which  repeatedly restarts the Newton iteration with random starting %values,
computes more than 40 solutions per second. %(not all of them different, of course).
 Remarkably,
for 3 of the obtained solutions some consecutive exponentials commute and can thus
be joined together, 
such that  only 8 exponentials remain, e.g.,
\begin{align*}
        \nS&=\ee^{f_{11}\AA_1-f_{12}\AA_2+f_{13}\AA_3-f_{14}\AA_4}
    \,\ee^{f_{21}\AA_1-f_{22}\AA_2+f_{23}\AA_3-f_{24}\AA_4}\\
        &\quad\times\ee^{f_{31}\AA_1-f_{32}\AA_2+f_{33}\AA_3-f_{34}\AA_4}
        \,\ee^{f_{41}\AA_1-f_{42}\AA_2+f_{43}\AA_3-f_{44}\AA_4}\\
        &\quad\times\ee^{f_{41}\AA_1+f_{42}\AA_2+f_{43}\AA_3+f_{44}\AA_4}
        \,\ee^{f_{31}\AA_1+f_{32}\AA_2+f_{33}\AA_3+f_{34}\AA_4}\\
        &\quad\times\ee^{f_{21}\AA_1+f_{22}\AA_2+f_{23}\AA_3+f_{24}\AA_4}
    \,\ee^{f_{11}\AA_1+f_{12}\AA_2+f_{13}\AA_3+f_{14}\AA_4}
\end{align*} 
with
\begin{align*}
&(f_{jk}) \doteq \\
%&{\fontsize{6.65}{7.3}\selectfont\left(\begin{array}{rrrr}
% 0.16808609092999572468  &  0.15127748183699615221 &  
% 0.11766026365099700728  &  0.06723443637199828987 \\
% 0.35936642058144077472  &  0.13138306991907331581 &
%-0.13090134825412630033  & -0.20289875692177817851 \\
% 0.40827036864282357822  & -0.23275549365763740519 & 
%-0.08579083407432252902  &  0.33387939732570943825 \\
%-0.43572288015426007762  &  0.24554796063280398470 &      
% 0.09903191867745182207  & -0.36832194094920597705
%\end{array}\right)},
&{\fontsize{7.2}{8}\selectfont\left(\begin{array}{rrrr}
 0.168086090929995725  &  0.151277481836996152 &  
 0.117660263650997007  &  0.067234436371998290 \\
 0.359366420581440775  &  0.131383069919073316 &
-0.130901348254126300  & -0.202898756921778179 \\
 0.408270368642823578  & -0.232755493657637405 & 
-0.085790834074322529  &  0.333879397325709438 \\
-0.435722880154260078  &  0.245547960632803985 &      
 0.099031918677451822  & -0.368321940949205977
\end{array}\right)}.
\end{align*}
For this scheme we obtain $\mathrm{LEM}\doteq 0.008976$, which is only a little larger than the lower bound 
$0.008956$ of  (\ref{eq:lem_bound}), and a little smaller than  the $\mathrm{LEM}\doteq 0.008999$ for
the scheme from \cite[Table~4]{alvfeh11}. 

Performing in this scheme the substitution (\ref{eq:Al_substitution}) 
with Gaussian nodes and weights
of order eight,
$$
(x_k) = \left(
\tfrac{1}{2}-\sqrt{\tfrac{15+2\sqrt{30}}{140}},\
\tfrac{1}{2}-\sqrt{\tfrac{15-2\sqrt{30}}{140}},\
\tfrac{1}{2}+\sqrt{\tfrac{15-2\sqrt{30}}{140}},\
\tfrac{1}{2}+\sqrt{\tfrac{15+2\sqrt{30}}{140}}
\right),
$$
$$
(w_k) =\left(
\tfrac{1}{4}-\tfrac{\sqrt{30}}{72},\
\tfrac{1}{4}+\tfrac{\sqrt{30}}{72},\
\tfrac{1}{4}+\tfrac{\sqrt{30}}{72},\
\tfrac{1}{4}-\tfrac{\sqrt{30}}{72}
\right),
$$ 
we obtain
$$\nS_A(t_n,\tau) = \prod_{j=8,\ldots,1}\exp\big(\tau\sum_{k=1}^4 a_{jk}A(t_n+\tau x_k)\big)
$$
with 
\begin{align*}
&(a_{jk}) \doteq \\
%&{\fontsize{6.65}{7.3}\selectfont\left(\begin{array}{rrrr}
% 0.18480846262431303905   &    -0.02072066212020042014 &
% 0.00502711867953985525   &    -0.00102882825365674947 \\
%-0.02344947887011890424   &     0.42125900994862326027 &     
%-0.04748789863325976613   &     0.00904478813619618483 \\
% 0.04462036092361700795   &    -0.21236935686571736948 &      
% 0.56998951780225396591   &     0.00602984678266997385 \\
%-0.04937525157353677699   &     0.23298947686588255412 &     
%-0.62261462824584900847   &     0.00327752279924315371 \\
% 0.00327752279924315371   &    -0.62261462824584900847 &      
% 0.23298947686588255412   &    -0.04937525157353677699 \\
% 0.00602984678266997385   &     0.56998951780225396591 &     
%-0.21236935686571736948   &     0.04462036092361700795 \\
% 0.00904478813619618483   &    -0.04748789863325976613 &      
% 0.42125900994862326027   &    -0.02344947887011890424 \\
%-0.00102882825365674947   &     0.00502711867953985525 &     
%-0.02072066212020042014   &     0.18480846262431303905
%\end{array}\right)}
&{\fontsize{7.2}{8}\selectfont\left(\begin{array}{rrrr}
 0.184808462624313039   &    -0.020720662120200420 &
 0.005027118679539855   &    -0.001028828253656749 \\
-0.023449478870118904   &     0.421259009948623260 &     
-0.047487898633259766   &     0.009044788136196185 \\
 0.044620360923617008   &    -0.212369356865717369 &      
 0.569989517802253966   &     0.006029846782669974 \\
-0.049375251573536777   &     0.232989476865882554 &     
-0.622614628245849008   &     0.003277522799243154 \\
 0.003277522799243154   &    -0.622614628245849008 &      
 0.232989476865882554   &    -0.049375251573536777 \\
 0.006029846782669974   &     0.569989517802253966 &     
-0.212369356865717369   &     0.044620360923617008 \\
 0.009044788136196185   &    -0.047487898633259766 &      
 0.421259009948623260   &    -0.023449478870118904 \\
-0.001028828253656749   &     0.005027118679539855 &     
-0.020720662120200420   &     0.184808462624313039
\end{array}\right)}.
\end{align*}

%\begin{acknowledgements}
%This work was supported by  the Austrian Science Fund (FWF) under Grant
%P 30819-N32.
%\end{acknowledgements}

%\bibliographystyle{plain}%{alpha}
%\bibliography{order_conditions}

\end{document}